\documentclass{amsart} 

\usepackage[english]{babel}
\usepackage{hyperref}
\usepackage{enumerate}
\usepackage{amsmath}
\usepackage{amsthm}
\usepackage{amssymb}
\usepackage{graphicx}
\usepackage{ucs}
\usepackage{enumerate}

\usepackage[utf8x,utf8]{inputenc}
\usepackage[T1]{fontenc}

\usepackage{todonotes}
 \usepackage{color}

  \newcommand{\N}{\mathbb N}
  \newcommand{\Q}{\mathbb Q}
  \newcommand{\Z}{\mathbb Z}
  \newcommand{\C}{\mathbb C}

  \newcommand{\inv}{^{-1}}

  \newcommand{\ilim}{\varprojlim}

  \newcommand{\Trd}{\mathcal{T}_d}
  
  \newcommand{\cw}{\mathrm{cw}}
  \renewcommand{\leq}{\leqslant}
  \renewcommand{\geq}{\geqslant}

  \newcommand{\act}{\curvearrowright}
  
  \newcommand{\grp}[1]{\langle #1\rangle}

  \newcommand{\conc}{^{\smallfrown}}
  \newcommand{\rist}{\mathrm{rist}}
    
  \newcommand{\impl}{\Rightarrow}

  \newcommand{\la}{\left\langle}
  \newcommand{\ra}{\right\rangle}
  \newcommand{\into}{\hookrightarrow}
  
  \newcommand{\normal}{\trianglelefteq}

\newcommand{\rest}{\upharpoonright}
\newcommand{\st}{\mathrm{st}}
\newcommand{\tor}{\mathrm{Tor}}

\newcommand{\ol}[1]{\overline{#1}}

\DeclareMathOperator{\supp}{\mathrm{supp}}

\DeclareMathOperator{\Aut}{\mathrm{Aut}}

\newcommand{\Sym}{\mathfrak{S}}

\newtheorem{thm}{Theorem}[section]
\newtheorem{cor}[thm]{Corollary}
\newtheorem{lem}[thm]{Lemma}
\newtheorem{prop}[thm]{Proposition}
\newtheorem{fact}[thm]{Fact}

\newtheorem{qu}[thm]{Question}

\theoremstyle{definition}

\newtheorem{df}[thm]{Definition}

\newtheorem{rmq}[thm]{Remark}
\newtheorem*{exemple}{Example}
\newtheorem{obs}[thm]{Observation}
\newtheorem*{ack}{Acknowledgments}

\title{On strongly just infinite profinite branch groups}

\author{François Le Maître}
\address{Institut de Mathématiques de Jussieu-PRG, Université Paris Diderot, Sorbonne Paris Cité, 
75205 Paris cedex 13, France}
   \email{francois.le-maitre@imj-prg.fr }
 
 \author{Phillip Wesolek}
 \address{Université catholique de Louvain,
    Institut de Recherche en Mathématiques et Physique (IRMP),
    Chemin du Cyclotron 2, box L7.01.02,
    1348 Louvain-la-Neuve, Belgique}
 \email{phillip.wesolek@uclouvain.be}

\begin{document}
\begin{abstract}
For profinite branch groups, we first demonstrate the equivalence of the Bergman property, uncountable cofinality, Cayley boundedness, the countable index property, and the condition that every non-trivial normal subgroup is open; compact groups enjoying the last condition are called strongly just infinite. For strongly just infinite profinite branch groups with mild additional assumptions, we verify the invariant automatic continuity property and the locally compact automatic continuity property. Examples are then presented, including the profinite completion of the first Grigorchuk group. As an application, we show that many Burger--Mozes universal simple groups enjoy several automatic continuity properties.
\end{abstract}

\maketitle

\renewcommand*{\thefootnote}{\arabic{footnote}}
\setcounter{footnote}{1}

\section{Introduction}
Given a Polish group\footnote{A \textbf{Polish group} is a separable topological group whose topology admits a compatible complete metric.}, it is natural to study the extent to which topological properties are determined by the group's algebraic structure.  One approach is to study the homomorphisms a group admits. Indeed, homomorphisms from a Polish group must respect algebraic structure, but they do not necessarily respect the topology. One thus considers automatic continuity questions: for a Polish group $G$ and a natural class of topological groups $\mathcal{H}$, one asks if \textit{every} homomorphism $\psi:G\rightarrow H$ with $H\in \mathcal{H}$ is continuous. Natural choices for $\mathcal{H}$ include the classes of Polish groups, non-archimedean Polish groups, small invariant neighborhood Polish groups, and locally compact Polish groups.

In the setting of non-locally compact Polish groups, there are now many groups known to enjoy the \textbf{automatic continuity property}, namely that every homomorphism into any Polish group is continuous; we refer the reader to the nice survey by C. Rosendal \cite{R09} for further discussion. For non-discrete locally compact Polish groups, however, much less is known. Indeed, the following fundamental question remains open:

\begin{qu}[Rosendal, \cite{MR2430439}]\label{qu:main}
Is there a non-discrete locally compact Polish group which has the automatic continuity property?
\end{qu}

In the work at hand, we study automatic continuity properties for profinite branch groups. Our results fall just short of answering Question~\ref{qu:main} positively; specifically, we obtain the weak Steinhaus property. Nonetheless, we do elucidate an interesting characterization of those groups which enjoy various weaker automatic continuity properties. Moreover, we connect these properties to combinatorial boundedness conditions. The strongest of these being the Bergman property.

\begin{df}
A group $G$ has the \textbf{Bergman property}\footnote{This is also sometimes referred to as \textit{strong uncountable cofinality} or \textit{strong boundedness}.} if every $G$-action by isometries on a metric space has bounded orbits.
\end{df}

When $G$ is a compact group, every \textit{continuous} $G$-action by isometries has bounded orbits. Upon relaxing the continuity requirement, however, it is not clear a priori that there exists an infinite group satisfying such a strong condition. Actually, the first infinite group which was shown to have this property is non-compact:  G. M. Bergman proved that the group of permutations of the integers has Bergman's property \cite{MR2239037}. Profinite examples were later exhibited independently by Y. Cornulier and A. Khelif; see \cite{MR2503307} for further examples and references. 

By work of Cornulier \cite{MR2240370}, the Bergman property is equivalent to the conjunction of two weaker boundedness properties: 
\begin{df}
A group $G$ has \textbf{uncountable cofinality} if there is no increasing chain $(G_n)_{n\in\N}$ of proper subgroups of $G$ such that $\bigcup_{n\in\N} G_n=G$. The group $G$ is \textbf{Cayley bounded} if for any finite symmetric generating set $U$ containing $1$, there is $n\geq 1$ so that $U^n=G$. Equivalently, every Cayley graph for $G$ has finite diameter.
\end{df}

The connection between automatic continuity properties and boundedness principles that we obtain is analogous to the following theorem of S. Thomas et al.:

\begin{thm}[Gyenis, Thomas, Saxl--Shelah--Thomas, Saxl--Wilson, cf. \cite{T99,TZ12}]\label{thm:products}
Let $(S_i)_{i\in \N}$ be a sequence of non-abelian finite simple groups and let $G=\prod_{i\in \N} S_i$. Then the following are equivalent.
\begin{enumerate}[(1)]
\item Every homomorphism from $G$ into a non-archimedean Polish group is continuous.
\item The sequence $(S_i)_{i\in \N}$ fails both the Malcev condition and the Saxl--Wilson condition.
\item $G$ has uncountable cofinality and the finite index property.
\item $G$ has uncountable cofinality and the normal countable index property.
\item $G$ has Bergman's property and the finite index property.
\item $G$ has Bergman's property and the normal countable index property.
\end{enumerate}
\end{thm}
\noindent A group has the \textbf{normal countable index property} if every homomorphism into a countable discrete group is continuous. A group has the \textbf{finite index property} if every homomorphisms into a finite discrete group is continuous. Equivalent $(1)$ is called the \textbf{countable index property}.

\begin{exemple}[{\cite[Theorem 1.10]{T99}}]The profinite group $\prod_{n\in\N} \mathrm{Alt}(n)$ enjoys the equivalent conditions of Theorem~\ref{thm:products}.
\end{exemple}
Theorem~\ref{thm:products} is striking because it not only identifies profinite groups with interesting automatic continuity properties but also uncovers a relationship between Bergman's property and several automatic continuity properties. Our analogue for profinite branch groups exhibits an even stronger relationship between combinatorial boundedness conditions and automatic continuity properties.

\subsection{Statement of results}
\begin{df}[See Section~\ref{sec:branch groups}] 
A \textbf{profinite branch group} is a closed spherically transitive subgroup of the automorphism group of a rooted locally finite tree such that every rigid level stabilizer is an open subgroup.
\end{df}
\noindent Profinite branch groups form a rich, interesting class of compact Polish groups; an introduction to these groups and their structure can be found in \cite{G00}.

A profinite group is \textbf{strongly just infinite} if every non-trivial normal subgroup is open. A straightforward adaption of an argument due to Grigorchuk, cf. \cite[Theorem 4]{G00}, gives a characterization of strongly just infinite profinite branch groups.
\begin{prop}[See Theorem~\ref{thm:sji_char}]
Suppose that $G\leq \Aut(T_{\alpha})$ is a profinite branch group. Then the following are equivalent:
\begin{enumerate}[(1)]
\item $G$ is strongly just infinite.
\item For every vertex $v\in T_{\alpha}$, the derived subgroup of $\rist_G(v)$ is open in $\rist_G(v)$.
\item For every level $n\geq 1$, the derived subgroup of $\rist_G(n)$ is open.
\end{enumerate}
\end{prop}

We then establish equivalences between being strongly just infinite, automatic continuity properties, and combinatorial boundedness conditions, giving an analogue of Theorem~\ref{thm:products}. 

\begin{thm}[See Theorem \ref{thm:equiv}]\label{thm:equiv_intro}
Let $G$ be a profinite branch group. Then the following are equivalent:
\begin{enumerate}[(1)]
\item $G$ is strongly just infinite.
\item Every commensurated subgroup of $G$ is either finite or open.
\item $G$ has the normal countable index property.
\item $G$ has the countable index property.
\item $G$ has the weak Steinhaus property.
\item $G$ has uncountable cofinality.
\item $G$ is Cayley-Bounded.
\item $G$ has property (FA).
\item $G$ has the Bergman property.
\end{enumerate}
\end{thm}

Examples of groups satisfying the equivalent conditions of Theorem~\ref{thm:equiv_intro} are given by many iterated wreath products as well as the closure of the Grigorchuk group. (See Section \ref{sec:examples}.) 

Under slightly stronger hypotheses, we obtain additional automatic continuity properties.

\begin{df}
A Polish group $G$ has the \textbf{invariant automatic continuity property} if every homomorphism $G\to H$ with $H$ a small invariant neighborhood Polish group is continuous. It has the \textbf{locally compact automatic continuity property} if every homomorphism $G\to H$ with $H$ a locally compact Polish group is continuous
\end{df}

\begin{thm}[See Theorem~\ref{thm:invariant_ACP}]\label{thm:invariant_ACP_intro}
If $G$ is a strongly just infinite profinite branch group which locally has derangements and has uniform commutator widths, then $G$ enjoys the equivalent properties of Theorem~\ref{thm:equiv_intro}, the invariant automatic automatic continuity property, and the locally compact automatic continuity property.
\end{thm}

We go on to observe that profinite branch groups have a unique Polish group topology. Additionally, strongly just infinite branch groups admit exactly two locally compact group topologies: their profinite topology and the discrete topology. (See Section \ref{sec:unique topology}.)


Our study of profinite branch groups concludes by considering applications of our results. We give the first examples of non-discrete compactly generated locally compact Polish groups that are simple and enjoy the automatic continuity properties discussed herein.

\begin{thm}[See Theorem~\ref{thm:BM_aut}]
Suppose that $d\geq 6$ and that $F\leq \Sym_d$ is perfect, two transitive, and generated by point stabilizers. Suppose further the point stabilizers of $F$ are also perfect. The Burger--Mozes universal simple group $U(F)^+$ then enjoys the countable index property, the invariant automatic continuity property, and the locally compact automatic continuity property.
\end{thm}

The commensurated subgroups of these Burger--Mozes groups are additionally classified. Classifying commensurated subgroups gives information on the possible homomorphisms into totally disconnected locally compact groups, see \cite{RW_Hom_15}; a compelling example of such a classification is the work of Y. Shalom and G. Willis on commensurated subgroups of arithmetic groups \cite{SW13}.

\begin{thm}[see Theorem~\ref{thm:BM_comm}]
Suppose that $d\geq 6$ and that $F\leq \Sym_d$ is perfect, two transitive, and generated by point stabilizers. Suppose further the point stabilizers of $F$ are also perfect. Then every commensurated subgroup of $U(F)^+$ is either finite, compact and open, or equal to $U(F)^+$.
\end{thm}
The alternating group $A_d$ for any $d\geq 6$ is an example of a finite group $F$ that satisfies the hypotheses of the above theorems.

\begin{ack} We would like to thank Pierre-Emmanuel Caprace for his many helpful comments around this project and Nikolay Nikolov for his detailed remarks on examples. We would also like to thank Anatole Khelif for allowing us to reproduce an unpublished proof of his in the present work.
	
François Le Maître was supported by Interuniversity Attraction Pole DYGEST and Projet ANR-14- CE25-0004 GAMME. 	Phillip Wesolek was supported by the ERC grant \#278469.
\end{ack}

\section{Preliminaries}
For a topological group $G$, the commutator of $g,h\in G$ is $[g,h]:=ghg^{-1}h^{-1}$. The \textit{set} of commutators of $G$ is $[G,G]:=\{[g,h]\mid g,h\in G\}$. We put 
\[
[G,G]^{\ast n}:=\{[g_1,h_1]\dots [g_n,h_n]\mid g_i,h_i\in G\}.
\]
The \textbf{derived subgroup} of $G$ is $D(G):=\grp{[G,G]}$; in general this subgroup is not closed, hence we occasionally add the modifier ``abstract'' to call attention to this point. 

The symmetric group on a set $X$ is denoted $\Sym(X)$. For all $d\in\N$, we let $[d]$ denote the set $\{0,...,d-1\}$, and we set $\Sym_d:= \Sym([d])$.

\subsection{Generalities on groups}
We shall require a lemma likely well-known to mathematicians more familiar with the theory of uncountable abelian groups; we include a proof for completeness. 

\begin{lem}\label{lem:countable_quotient}
If $A$ is an uncountable abelian group, then $A$ has an infinitely generated countable quotient.
\end{lem}
\begin{proof}
Let $\tor(A)$ be the torsion subgroup of $A$ and form $\tilde{A}:=A/\tor(A)$. Suppose first $\tilde{A}$ is uncountable, so $\tilde{A}$ is an uncountable torsion free abelian group. The extension of scalars $\tilde{A}\otimes_{\Z} \Q$ is an uncountable $\Q$-vector space, and there is a canonical injection $\tilde{A}\into \tilde{A}\otimes_{\Z} \Q$, since $\tilde{A}$ is torsion free. We may find $\{a_i\otimes 1\}_{i\in \N}$ with $a_i\in \tilde{A}$ linearly independent vectors in $\tilde{A}\otimes_{\Z} \Q$. We then have a projection  
\[
\tilde{A}\otimes_{\Z} \Q\rightarrow \mathrm{span}(\{a_i\otimes 1\mid i\in \N\}).
\]
The composition $A\rightarrow\tilde{A}\rightarrow V$ has a countably infinite image that is infinitely generated verifying the lemma in this case. 

If $A/\tor(A)$ is countable, then it suffices to find a countable quotient of $\tor(A)$ that is infinitely generated; we thus assume $A=\tor(A)$. By \cite[Theorem 8.4]{F70}, we have a decomposition   $A=\bigoplus_{p \text{ prime}} A_p$ where $A_p$ are abelian $p$-groups, and since $A$ is uncountable, there is a prime $p$ so that $A_p$ is uncountable. We may thus also assume $A$ is an uncountable $p$-group.

Appealing to \cite[Theorem 32.3]{F70}, there is $B\leq A$ so that $B$ is a direct sum of cyclic $p$-groups and $A/B$ is divisible. Suppose first $A/B$ is non-trivial. Divisible abelian groups are direct sums of copies of $\Q$ and Pr\"{u}fer $p$-groups via \cite[Theorem 23.1]{F70}, and both of these are countably generated. Projecting onto one of these summands, we obtain a countable quotient of $A$ which is infinity generated. If $A/B$ is trivial, then $A$ is an uncountable direct sum of cyclic $p$-groups. Projecting onto a countable direct sum gives the desired countable quotient which is infinitely generated. The lemma is thus verified.  
\end{proof}

For a group $G$, subgroups $H$ and $K$ are \textbf{commensurate} if $|H:H\cap K|$ and $|K:H\cap K|$ are both finite. It is easy to check the relation of commensurability is an equivalence relation on subgroups of $G$. We denote this equivalence relation by $\sim_c$. For $H\leq G$, the \textbf{commensurator} of $H$ is $G$ is defined to be
\[
Comm_G(H):=\left\{ g\in G \mid H\sim_c gHg^{-1}\right\};
\]
the set $Comm_G(H)$ is indeed a subgroup of $G$. We say $H\leq G$ is \textbf{commensurated} if $Comm_G(H)=G$.

We shall require an important feature of commensurated subgroups:

\begin{thm}[Bergman--Lenstra, {\cite[Theorem 6]{BL89}}]\label{thm:BL}
Let $G$ be a group with subgroups $H$ and $K$. Then the following are equivalent:
\begin{enumerate}[(1)]
\item $\sup_{k\in K}|H:H\cap kHk^{-1}|<\infty$.
\item There is $N$ normalized by $K$ so that $N\sim_c H$.
\end{enumerate}
\end{thm}

For a group $G$, define the function $l:D(G)\rightarrow \Z$ by
\[
l(g):=\min\{n \mid g=\prod_{i=1}^n[h_i,k_i] \}.
\]
The \textbf{commutator width} of $G$ is then $\cw(G):= l(D(G))$.

For profinite groups, we note a useful sufficient condition for finite commutator width.
\begin{lem}[folklore]\label{lem:commutator width}
If $G$ is a profinite group so that $D(G)$ is open in $G$, then $\cw(G)<\infty$. 
\end{lem}

\begin{proof}
The abstract derived subgroup is so that $D(G)=\bigcup_{k\geq 0}[G,G]^{\ast k}$, and the sets $[G,G]^{\ast k}$ are closed. The Baire category theorem thus implies $[G,G]^{\ast k}$ has nonempty interior for some $k$. Since $D(G)$ is also compact, there is some finite set $A\subseteq D(G)$ so that $A[G,G]^{\ast k}=D(G)$. It now follows there is $N\geq k$ so that $D(G)=[G,G]^{\ast N}$. That is to say, $\cw(G)\leq N$.
\end{proof}

Lastly, a topological group is called \textbf{just infinite} if all its non-trivial closed normal subgroups are of finite index. Relaxing the closedness requirement, we arrive at the central notion of the current work.
\begin{df}
A topological group $G$ is called \textbf{strongly just infinite} if every nontrivial normal subgroup is open with finite index. 
\end{df}

\subsection{Combinatorial boundedness conditions}\label{subsec:dfbergman}
Recall a group $G$ has the Bergman property if every $G$-action by isometries on a metric space has bounded orbits. The Bergman property admits a useful algebraic reformulation, which we will use repeatedly.

\begin{df}
A sequence $(A_n)_{n\in \N}$ of subsets of a group $G$ is called a \textbf{Bergman sequence} if it is increasing, each of its elements is symmetric, $1\in A_0$, and $\bigcup_{n\in\N} A_n=G$.
\end{df}

\begin{thm}[Cornulier, cf. {\cite[Proposition 2.7]{C06}}]
Let $G$ be a group. The following assertions are equivalent.
\begin{enumerate}[(1)]
\item $G$ has the Bergman property.
\item If $(A_n)_{n\in \N}$ is a Bergman sequence such that $A_nA_n\subseteq A_{n+1}$ for all $n\in \N$, then there exists $k\in\N$ such that $A_k=G$.
\item If $(A_n)_{n\in \N}$ is a Bergman sequence, then there exists $k,n\in\N$ such that $A_n^k=G$.
\end{enumerate}
\end{thm}
%
%
%

As noted in the introduction, the Bergman property is closely related to two weaker boundedness conditions. 

\begin{df}
A group $G$ has \textbf{uncountable cofinality} if there is no increasing chain $(G_n)_{n\in\N}$ of proper subgroups of $G$ such that $\bigcup_{n\in\N} G_n=G$. Otherwise, $G$ has \textbf{countable cofinality}.
\end{df}
\begin{df}
A group $G$ is \textbf{Cayley bounded} if for any symmetric generating set $U$ containing $1$, there is $n\geq 1$ so that $U^n=G$. Equivalently, every Cayley graph for $G$ has finite diameter.
\end{df}
We will use repeatedly, and without comment, the easily established fact that any quotient of a Cayley bounded group is again Cayley bounded.

\begin{prop}[{Cornulier, \cite[Proposition 2.4]{MR2240370}}]\label{prop:bergman_cof,cay}
A group $G$ has the Bergman property if and only if it has uncountable cofinality and is Cayley bounded
\end{prop}

No countable group has the Bergman property; however it is unknown whether every infinite countable group fails to be Cayley bounded. 

The following useful results were announced by A. Khelif in \cite{K06}. Here we reproduce his unpublished proofs with his kind permission. Note that his terminology differs from ours: what he calls ``Bergman's property'' is for us ``Cayley boundedness.''

\begin{thm}[Khelif, {\cite[Th\'{e}or\`{e}me 10]{K06}}]\label{thm:countable infinite in compact is not Cb}
	If $\Gamma$ is a countable infinite subgroup of a compact group $G$, then $\Gamma$ is not Cayley bounded. 
\end{thm}
\begin{proof}(Khelif)
First note that we may assume $G$ is metrizable: By the Peter-Weyl Theorem, for each $\gamma\in\Gamma$ there is a finite dimensional unitary representation $\pi_\gamma: G\to \mathcal U(\C^{n_\gamma})$ such that $\pi_\gamma(\gamma)\neq 1$, so $\Gamma$ embeds into the compact metrizable group $\prod_{\gamma\in\Gamma} \mathcal U(\C^{n_\gamma})$. Furthermore, by taking the closure of $\Gamma$ in $G$, we may assume that $\Gamma$ is dense in $G$. Fix a compatible right-invariant metric $d$ on $G$. 

Enumerate $\Gamma=\{\gamma_n: n\in\N\}$ and for all $n\in\N$ let $\Gamma_n:=\la \gamma_0,\dots,\gamma_n\ra$. Consider the sequence of continuous functions $f_n:G\to[0,+\infty[$ given by 
\[
f_n(g)=d(\Gamma_n,g):=\inf_{\gamma\in\Gamma_n}d(\gamma,g).
\]
This sequence of functions decreases pointwise to zero, so by Dini's theorem, they converge uniformly to zero. We can thus find a sequence $(u_n)_{n\in \N}$ such that $u_n\in\Gamma_n$ and $d(u_n,\gamma_{n+1})$ tends to zero. 

Set $S:=\{1,\gamma_0\}\cup\{\gamma_{n+1}u_n\inv \mid n\in\N\}$ and let $U:=S\cup S\inv$. The set $U$ is a symmetric generating set for $\Gamma$ which contains $1$, and since $\gamma_{n+1}u_n\inv\to 1$, the set $U$ is also a compact subset of $G$. If $\Gamma$ is Cayley bounded, we can find $k\in\N$ such that $\Gamma=U^k$. In particular, $\Gamma$ is then compact, but by the Baire category theorem, there is no countable infinite compact group. We thus deduce that $\Gamma$ is not Cayley bounded.
\end{proof}

\begin{cor}[{\cite[Corollaire 11]{K06}}]\label{cor:ab_not_cay}
If $G$ is an infinite solvable-by-finite group, then $G$ is not Cayley bounded.
\end{cor}
\begin{proof}(Khelif)
Let $H$ be a finite index normal subgroup of $G$ such that $H$ is solvable and consider the derived series $H^{(n+1)}:=D(H^{(n)})$ where $H^{(0)}:=H$. Let $n$ be the smallest integer such that $H/H^{(n)}$ is infinite. Since $H^{(n-1)}/H^{(n)}$ is abelian, the quotient group $G/H^{(n)}$ is infinite and abelian-by-finite. It thus suffices to show that no infinite abelian-by-finite group is Cayley bounded, so we assume $G$ has a finite index normal subgroup $N$ which is abelian. 

By Lemma \ref{lem:countable_quotient}, the group $N$ has a subgroup $B\leq N$ with countable index. Letting $x_1,\dots,x_k$ be coset representatives for $N$ in $G$, the subgroup $L:=\bigcap_{i=1}^kx_iBx_i^{-1}$ is a normal subgroup of $G$ so that $G/L$ is a countably infinite abelian-by-finite group. Passing to $G/L$, we may also assume $G$ is countable.

The Pontryagin dual of $N$ may be seen as a discrete group whose dual is a compact group $K$ into which $N$ embeds. Hence, we obtain a morphism $\psi:N\to K$ with countable infinite image where $K$ is a compact group. Taking coset representatives $x_1,\dots,x_k$ for $N$ in $G$, define $\rho:N\rightarrow K^{k}$ by 
\[
h\mapsto (\psi(x_1hx_{1}^{-1}),\dots,\psi(x_khx_k^{-1})).
\]
The action of $G$ on $N$ by conjugation induces an action of $G$ on $\rho(N)$ by defining $g.\rho(h):=\rho(ghg^{-1})$.

For each $g\in G$ and $h\in N$, it follows there are $h_i'\in N$ and $\sigma\in \Sym(k)$ so that
\[
g.(\psi(x_1hx_{1}^{-1}),\dots,\psi(x_khx_k^{-1}))=(\psi(x_{\sigma(1)}h_1'hh_1'^{-1}x_{\sigma(1)}^{-1}),\dots,\psi(x_{\sigma(k)}h'_khh'^{-}_kx_{\sigma(k)}^{-1}))
\]
Each element $g\in G$ thus acts on $\rho(N)$ by conjugating via elements of $\psi(N)$ and permuting coordinates. Such an action naturally extends to an action on $\ol{\rho(N)}$ by continuous automorphisms.

Under this action, the group $\ol{\rho(N)}\rtimes G$ is a locally compact group. A straightforward calculation shows the subset $\{(\rho(h),h^{-1})\mid h\in N\}$ is a discrete normal subgroup of $\ol{\rho(N)}\rtimes G$. The quotient is furthermore a compact group into which $G$ embeds. The desired conclusion now follows via Theorem \ref{thm:countable infinite in compact is not Cb}.
\end{proof}

We shall consider one final boundedness condition, which appears in J-P. Serre's fundamental work on groups and trees.
\begin{df}
A group $G$ satisfies \textbf{property (FA)} if whenever $G$ acts on a tree without edge inversions, then it fixes a vertex. 
\end{df}

\begin{thm}[Serre, \cite{MR0476875}]A group $G$ has property (FA) if an only if it has uncountable cofinality, it does not split as a nontrivial amalgamated product, and it does not surject onto $\Z$. 
\end{thm}

H. Bass shows that no profinite group splits as an amalgamated free product and that no profinite group surjects onto $\Z$ and thus deduces the following result: 

\begin{thm}[Bass, \cite{MR0419616}]\label{thm:FA}
A profinite group $G$ has property (FA) if an only if it has uncountable cofinality.
\end{thm}

\subsection{Automatic continuity properties}
\begin{df}
A topological group is \textbf{Polish} if the underlying topology is separable and admits a compatible  complete metric. A Polish group is called \textbf{non-archimedean} if the topology admits a basis at $1$ of open subgroups. It is called small invariant neighborhood or \textbf{SIN} if it admits a basis at $1$ of conjugation invariant sets.
\end{df}

If $G$ is a Polish group which admits a compatible two-sided invariant metric, then the metric balls around $1$ for such a metric witness that $G$ is a SIN group. Conversely, the construction in the Birkhoff-Kakutani metrization theorem ensures that any SIN Polish group has a compatible two-sided invariant metric. SIN Polish groups are thus sometimes referred to as TSI groups. By integrating a compatible left-invariant metric on a compact metrizable group against the Haar measure, we obtain a two-sided invariant metric, so every compact metrizable group is a SIN group.

The first two automatic continuity properties of interest have a useful algebraic description.
\begin{df}Let $G$ be a Polish group.
\begin{enumerate}[(i)]
\item The group $G$ has the \textbf{normal countable index property} if every countable index normal subgroup of $G$ is open.
\item The group $G$ has the \textbf{countable index property} if every countable index  subgroup of $G$ is open.
\end{enumerate}
\end{df}

\begin{obs}Let $G$ be a Polish group.
\begin{enumerate}[(1)]
\item The group $G$ has the normal countable index property if and only if every homomorphism $\psi:G\rightarrow H$ with $H$ a countable discrete group is continuous.
\item The group $G$ has the countable index property if and only if every homomorphism $\psi:G\rightarrow H$ with $H$ a non-archimedean Polish group is continuous.
\end{enumerate}
\end{obs}
The countable index property is a weakening of the small index property which states that every subgroup of index less than the continuum is open. The latter has been studied in the context of automorphism groups of countable structures; see  \cite{zbMATH03981411}.

We shall also consider two further automatic continuity properties.
\begin{df}
Let $G$ be a Polish group. 
\begin{enumerate}[(i)]
\item The group $G$ has the \textbf{invariant automatic continuity property} if every homomorphism $\psi:G\rightarrow H$ with $H$ a SIN Polish group is continuous.
\item The group $G$ has the \textbf{locally compact automatic continuity property} if every homomorphism $\psi:G\rightarrow H$ with $H$ a locally compact Polish group is continuous.
\end{enumerate}
\end{df}
We observe that the invariant automatic continuity property and the locally compact automatic continuity property have a relationship in the presence of the Bergman property.
\begin{lem}\label{lem: iac+bergman implies lcac}
Let $G$ be a Polish group. If $G$ has the invariant automatic continuity property and the Bergman property, then $G$ has the locally compact automatic continuity property.
\end{lem}
\begin{proof}
Let $H$ be a locally compact Polish group and $\varphi: G\to H$ a homomorphism. Since $H$ is locally compact and Polish, it is $\sigma$-compact, so we may write $H=\bigcup_{n\in \N} K_n$ where $(K_n)_{n\in \N}$ is an increasing sequence of compact subsets. We may  assume that $1\in K_0$ and up to replacing $K_n$ by $K_n\cup K_n\inv$ we may also assume that each $K_n$ is symmetric.

The sequence $(\varphi\inv(K_n))_{n\in\N}$ is a Bergman sequence in $G$, so there exists $k,n\in\N$ such that $\varphi\inv(K_n)^k=G$.  We deduce that $\varphi(G)\subseteq K_n^k$. Thus, $\varphi(G)$ has compact closure, and since compact groups are SIN groups, the conclusion follows from the invariant automatic continuity property.
\end{proof}

Let us now present the main algebraic tools behind automatic continuity proofs. For a Polish group $G$, a subset $A\subseteq G$ is called \textbf{$\sigma$-syndetic} if there is a sequence $(g_n)_{n\in\N}$ of elements of $G$ so that $G=\bigcup_{n\in\N}g_nA$. 

\begin{df}A Polish group $G$ has the \textbf{Steinhaus property} if there is $N>0$ such that for any symmetric $\sigma$-syndetic set $A$, the set $A^N$ contains a neighborhood of $1$.
\end{df}

\begin{prop}[Rosendal--Solecki, \cite{MR2365867}] If a Polish group $G$ has the Steinhaus property, then it has the automatic continuity property: every homomorphism $G\to H$ with $H$ a Polish group is continuous.
\end{prop}

A subset $A\subseteq G$ is called \textbf{invariant} if $gAg^{-1}=A$ for all $g\in G$. We will use the following variation: 

\begin{df}
A Polish group $G$ has the \textbf{invariant Steinhaus property} if there is $N>0$ such that for any symmetric invariant $\sigma$-syndetic set $A$, the set $A^N$ contains a neighborhood of $1$.
\end{df}

\begin{prop}[Dowerk--Thom, {\cite{Dowerk:2015yu}}]\label{prop:invariant_steinhaus}
If a Polish group $G$ has the invariant Steinhaus property, then $G$ has the invariant automatic continuity property.
\end{prop}
\begin{rmq}
In \cite{Dowerk:2015yu}, Ph. Dowerk and A. Thom show finite-dimensional unitary groups satisfy the invariant automatic continuity property, but these groups fail the automatic continuity property by a result of Kallman \cite{K00}. Of all the automatic continuity properties discussed here, invariant automatic continuity is \textit{the only} known to be strictly weaker than the automatic continuity property.
\end{rmq}

One further variation on the Steinhaus property will appear in the present work. \begin{df}
A Polish group $G$ has the \textbf{weak Steinhaus property} if for any symmetric $\sigma$-syndetic set $A$, there exists $n\in\N$ such that the set $A^n$ contains a neighborhood of $1$.
\end{df}
Note that the above condition is still a strengthening of the countable index property.
 We conclude our discussion with two technical results concerning $\sigma$-syndetic sets.

\begin{lem}[{\cite[Lemma 4]{Tsankov:2011lr}}]\label{lem:sigma syndetic subgroup} Suppose that $G$ is a group and $H$ is a subgroup of $G$. If $A$ is a $\sigma$-syndetic set for $G$, then $H\cap A^2$ is $\sigma$-syndetic for $H$.
\end{lem}

\begin{lem}\label{lem:sigma_sigma}
Suppose $G$ is a group and $g_0,\dots,g_n\in G$. If $A$ is a symmetric $\sigma$-syndetic set containing $1$, then $\bigcap_{i=0}^ng_iA^{2^n}g_i^{-1}$ is symmetric $\sigma$-syndetic set containing $1$.
\end{lem}
\begin{proof}
The fact that $\bigcap_{i=0}^ng_iA^{2^n}g_i^{-1}$ is again symmetric and contains $1$ is straightforward. To see this set is also $\sigma$-syndetic, one argues via the obvious induction. This induction comes down to the following claim, which we prove here: If $A,B\subseteq G$ are two $\sigma$-syndetic symmetric set containing $1$, then the set $B^2\cap A^2$ is also $\sigma$-syndetic.

Let $A$, $B$ be two $\sigma$-syndetic symmetric set containing $1$, fix a sequence $(h_k)_{k\in \N}$ so that $G=\bigcup_{k\in \N}h_kB$ and set 
\[
K:=\{k\in \N\mid h_kB\cap A\neq \emptyset\}.
\]
For each $k\in K$, fix $a_k\in h_kB\cap A$ and say $a_k=h_kb_k$ with $b_k\in B$.

For $a\in A$, there is $k\in K$ so that $a=h_kb$ for some $b\in B$, hence $a=a_kb_k^{-1}b$. We also have $a=a_ka_k^{-1}a$, so $a\in a_k( B^2\cap A^2)$. It now follows $A\subseteq \bigcup_{k\in K}a_k(B^2\cap A^2)$. Since the set $A$ is $\sigma$-syndetic, we conclude $B^2\cap A^2$ is $\sigma$-syndetic, completing the proof.
\end{proof}

\subsection{Rooted trees and profinite branch groups}\label{sec:branch groups}
A \textbf{rooted tree} $T$ is a locally finite tree with a distinguished vertex $r$ called the \textbf{root}. Letting $d$ be the usual graph metric, the \textbf{levels} of $T$ are the sets $V_n:=\{v\in T\mid d(v,r)=n\}$. The \textbf{degree} of a vertex $v\in V_n$ is the number of $w\in V_{n+1}$ so that there is an edge from $v$ to $w$; the collection of such $w$ are called the \textbf{children} of $v$. 

We think of the levels of the tree as linearly ordered so that the orders cohere. That is to say, if $v_0<v_1$ in $V_n$ with $w_0$ a child of $v_0$ and $w_1$ a child of $v_1$, then $w_0<w_1$ in $V_{n+1}$. This ordering allows us to take the right-most branch: the \textbf{right-most branch} of $T$ is the unique infinite path from the root $(v_i)_{i\in \N}$ so that $v_i$ is the maximal element of $V_i$ for all $i\in \N$.

When vertices $k$ and $w$ lie on the same path to the root and $d(k,r)\leq d(w,r)$, we write $k\leq w$. Given a vertex $s\in T$, the \textbf{tree below $s$}, denoted $T^s$, is the collection of $t$ so that $s\leq t$ along with the induced graph structure.

\begin{figure}[h]
    \centering
    \includegraphics[width=0.8\textwidth]{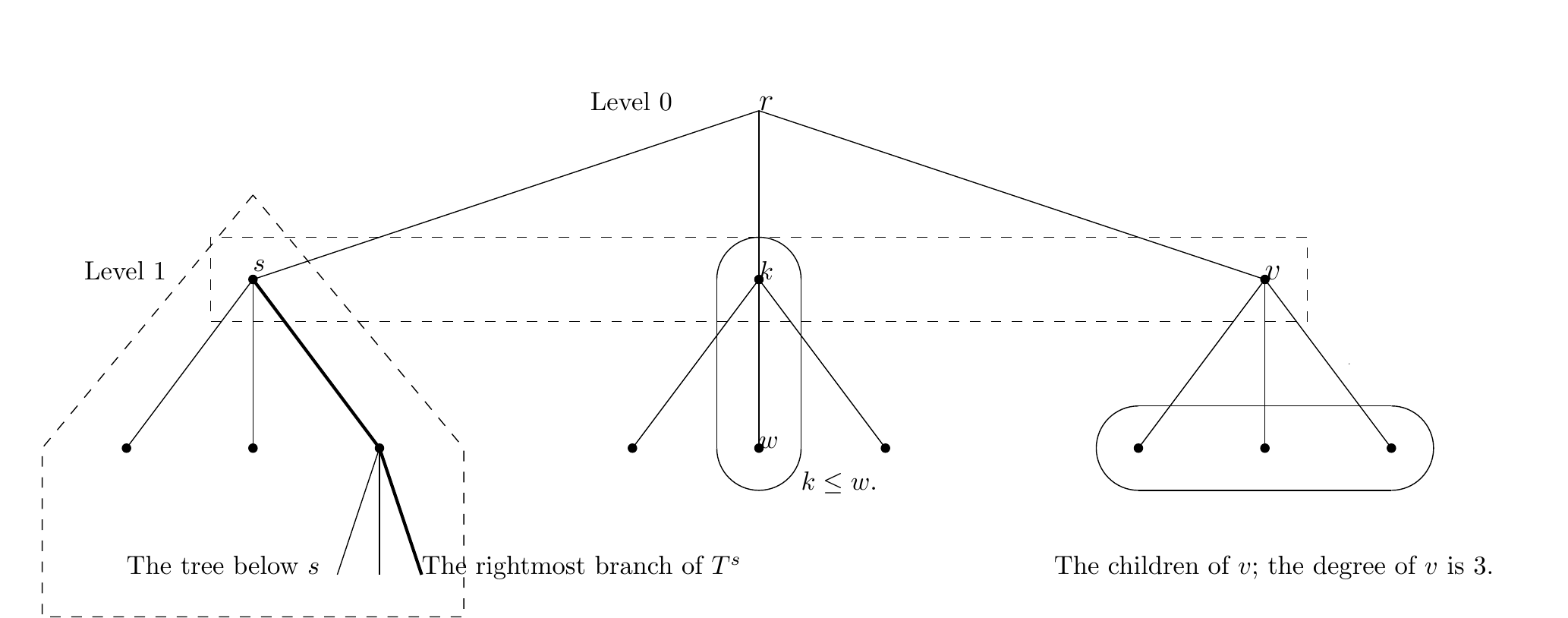}
    \caption{Levels, children, and trees below vertices}\label{fig:Tree_basics_1}
\end{figure}

We call a rooted tree \textbf{spherically homogeneous} if all $v,w\in V_n$ the degree of $v$ is the same as the degree of $w$; alternatively, $v$ and $w$ have the same number of children. A spherically homogeneous tree is completely determined by specifying the degree of the vertices at each level. These data are given by an infinite sequence $\alpha\in \N^{\N}$ so that $\alpha(i)\geq 2$ for all $i\in \N$;  the condition $\alpha(i)\geq 2$ ensures non-triviality. We denote a spherically homogeneous tree by $T_{\alpha}$ for $\alpha\in \N_{\geq 2}^{\N}$. When $\alpha\equiv d$, we write $T_d$. The group of rooted tree automorphisms, denoted $\Aut(T_\alpha)$, is naturally a profinite group. 

Profinite branch groups are certain closed subgroups of $\Aut(T_{\alpha})$; our approach to branch groups follows closely Grigorchuk's presentation in \cite{G00}. For $G\leq \Aut(T_{\alpha})$ a closed subgroup and for a vertex $v\in T_{\alpha}$, the \textbf{rigid stabilizer of $v$} in $G$ is defined to be
\[
\rist_G(v):=\{g\in G\mid g.w=w \text{ for all }w\in T_{\alpha}\setminus T_{\alpha}^v\}.
\]
The rigid stabilizer acts non-trivially only on the subtree $T^v_{\alpha}$. 

The \textbf{$n$-th rigid level stabilizer} in $G$ is defined to be
\[
\rist_G(n):=\grp{\rist_G(v)\mid v\in V_n}.
\]
It is easy to see that $\rist_G(n)\simeq \prod_{v\in V_n}\rist_G(v)$, and as a consequence, $\rist_G(n)$ is a closed subgroup of $G$. 

For a level $n$, we denote the pointwise stabilizer in $G$ of $V_n$ by $\st_G(n)$. The subgroup $\st_G(n)$ is called the \textbf{n-th level stabilizer} of $G$. Observe that it can be the case that $\rist_G(n)<\st_G(n)$, even for profinite branch groups.

\begin{df}\label{df:profinitebranch}
A profinite group $G$ is said to be a \textbf{profinite branch group} if there is a tree $T_{\alpha}$ for some $\alpha\in \N_{\geq 2}^{\N}$ so that the following hold:
\begin{enumerate}[(i)]
\item $G$ is isomorphic to a closed subgroup of $\Aut(T_{\alpha})$. 
\item $G$ acts transitively on each level of $T_{\alpha}$.
\item For each level $n$, the index $|G:\rist_G(n)|$ is finite.
\end{enumerate}
\end{df}
\noindent We shall always identify a profinite branch group $G$ with the isomorphic closed subgroup of $\Aut(T_{\alpha})$. 

The rigid level stabilizers form a basis at $1$ for the topology on a profinite branch group $G$. The transitivity of the action on the levels ensures that $\rist_G(v)\simeq \rist_G(w)$ for all $v,w\in V_n$. The transitivity further insures that profinite branch groups are always infinite.

\begin{lem}\label{lem:no_centre}
Suppose that $G\leq \Aut(T_{\alpha})$ is a profinite branch group and $v\in T_{\alpha}$. Then each $g\in C_{G}(\rist_G(v))$ fixes pointwise $T^v_{\alpha}$. In particular,  the center of $\rist_G(v)$ is trivial for all $v\in T_{\alpha}$.
\end{lem}
\begin{proof}
Fix $v\in T_{\alpha}$ and suppose for contradiction there is $w\geq v$ so that $g.w\neq w$. The subgroup $\rist_G(w)$ is non-trivial since $G$ is infinite, so we may find $y\in \rist_G(w)$ and $u\geq w$ with $y.u\neq u$.  The element $gy$ then sends $u$ to $g.(y.u)$, but the element $yg$ sends $u$ to $g.u$. Hence, $gy\neq yg$ contradicting that $g$ centralizes $\rist_G(v)$.
\end{proof}

As an immediate corollary, we obtain a description of $\rist_G(v)$.
\begin{lem}\label{lem:description_rist}
Suppose $G\leq \Aut(T_{\alpha})$ is a profinite branch group. If $v\in T_{\alpha}$ is at level $n$, then
\[
\rist_G(v)=\bigcap_{w\in V_n\setminus\{v\}}C_G(\rist_G(w)).
\]
\end{lem}

\section{Commutators and diagonalization in branch groups}\label{sec:comm and diago}

\subsection{The commutator trick}\label{sec:commutrick}

The following lemma is well-known and dates back to the fifties, where it was used by G. Higman \cite{MR0072136} to show the simplicity of various permutation groups. It has since become a cornerstone to proofs that various groups are simple.

The \textbf{support} of a permutation $\sigma\in\mathfrak S(X)$ is the set
\[
\supp (\sigma):=\{x\in X:\sigma(x)\neq x\}.
\]

\begin{lem}\label{lem:comm_trick}
Let $X$ be a set and let $G\leq \mathfrak S(X)$ be a permutation group with $\tau\in G$. If $\sigma_1,\sigma_2\in G$ are so that $\tau(\supp (\sigma_1))$ is disjoint from $\supp (\sigma_1)\cup \supp (\sigma_2)$, then the commutator $[\sigma_1,\sigma_2]$ is the product of four conjugates of $\tau^{\pm1}$ by elements of $G$. 
\end{lem}
\begin{proof}
Whenever two permutations have disjoint support, they commute. Moreover, 
\[
\tau(\supp \sigma_1)=\supp (\tau\sigma_1\tau\inv),
\]
so by our hypothesis, $\tau \sigma_1\tau\inv$ commutes with both $\sigma_1$ and $\sigma_2$. It follows $\tau\sigma_1\inv \tau\inv$ also commutes with both $\sigma_1$ and $\sigma_2$.

Setting $\bar{\sigma_1}:=[\sigma_1,\tau]=\sigma_1(\tau\sigma_1\inv\tau\inv)$, the fact that $\tau\sigma_1\inv\tau\inv$ commutes with both $\sigma_1$ and $\sigma_2$ yields that $[\sigma_1,\sigma_2]=[\bar \sigma_1,\sigma_2]$. The permutation $\bar \sigma_1=(\sigma_1\tau\sigma_1\inv)\tau\inv$ is the product of two conjugates of $\tau^{\pm1}$ by elements of the group $G$. Hence, $[\bar \sigma_1,\sigma_2]$ is the product of four conjugates of $\tau^{\pm1}$ by elements of $G$, verifying the lemma.
\end{proof}

We now adapt Grigorchuk's argument \cite[Theorem 4]{G00} to characterize strongly just infinite profinite branch groups. 

\begin{thm}\label{thm:sji_char}
Suppose that $G\leq \Aut(T_{\alpha})$ is a profinite branch group. Then the following are equivalent:
\begin{enumerate}[(1)]
\item $G$ is strongly just infinite;
\item For all $v\in T_{\alpha}$, the abstract derived subgroup $D(\rist_G(v))$ is open in $\rist_G(v)$.
\item For all $n\geq 1$, the abstract derived subgroup $D(\rist_G(n))$ is open in $G$.
\end{enumerate}
\end{thm}
\begin{proof}
For $(1)\Rightarrow (2)$, we prove the contrapositive. Suppose for some $v\in T_{\alpha}$ the abstract commutator $D(\rist_G(v))$ is not open in $\rist_G(v)$; Lemma~\ref{lem:no_centre} ensures $D(\rist_G(v))$ is also non-trivial.  Letting $n$ be the level of $v$, we have $D(\rist_G(n))=\prod_{w\in V_n}D(\rist_G(w))$ is a characteristic subgroup of $\rist_G(n)$ which is non-trivial and not open. Since $\rist_G(n)$ is a normal subgroup of $G$, we deduce that $G$ has a non-trivial normal subgroup which is not open, hence $G$ is not strongly just infinite.

\

The implication $(2)\Rightarrow (3)$ is immediate.

\

For $(3)\Rightarrow (1)$, let $H$ be a non-trivial normal subgroup of $G$ and let $\tau\in H\setminus\{1\}$. There exists a vertex $v$ such that $\tau(v)\neq v$; let $n$ be its level. Taking two elements $\sigma_1,\sigma_2$ in $\rist_G(v)$, their support is a subset of $T_{\alpha}^v$, and since $\tau(T_{\alpha}^v)=T_{\alpha}^{\tau(v)}$ is disjoint from $T_{\alpha}^v$, we apply the commutator trick and deduce that the commutator $[\sigma_1,\sigma_2]$ is the product of four conjugates of $\tau^{\pm1}$. 

The commutator group $D(\rist_G(v))$ is thus a subgroup of $H$. Since $H$ is normal and $G$ acts spherically transitively on $T_{\alpha}$, it follows the open subgroup $D(\rist_G(n))=\prod_{w\in V_n} D(\rist_G(w))$ is a subgroup of $H$, hence $H$ is open. 
\end{proof}

The next lemma establishes a version of the commutator trick for certain large sets.
\begin{df}
For $G\leq \Aut(T_{\alpha})$, we call a subset $A$ of $G$ \textbf{full above} the vertex $v\in T_{\alpha}$ if every element of $\rist_G(v)$ coincides with an element of $A$ restricted to $T_{\alpha}^v$. 
\end{df}
\noindent For a group $G$ with $\rist_G(v)=\{1\}$, any set containing $1$ is full above $v$. In profinite branch groups, however, rigid stabilizers are \textit{necessarily} infinite, so this trivial case never occurs.

The relevance of this definition stems from the following observation.

\begin{lem}\label{lem:comm_width}
Suppose that $G\leq \Aut(T_{\alpha})$ is a strongly just infinite profinite branch group and that $A\subseteq G$ is full above $v$. If $A\cap\rist_G(v)\nsubseteq \{1\}$, then there is $w\geq v$ so that $D(\rist_G(w))\leq A^{10k}$ with $k:=\cw(\rist_G(w))$.
\end{lem}
\begin{proof} 
Take $x\in A\cap\rist_G(v)\setminus \{1\}$ and find $w\geq v$ so that $x.w\neq w$. We now consider $g,h\in \rist_G(w)$. Since $A$ is full above $v$, there are $\tilde{g},\tilde{h}\in A$ with the same action on $T_{\alpha}^v$ as $g$ and $h$, respectively. The element $x$ is supported on $\rist_G(v)$, so $\tilde{h}x^{-1}\tilde{h}^{-1}=hx^{-1}h^{-1}$. In particular, $[\tilde{h},x]=[h,x]$. The element $[h,x]$ is again supported on $T_{\alpha}^v$, so we have that $[\tilde{g},[h,x]]=[g,[h,x]]$. 

Now $xh^{-1}x^{-1}$ commutes with both $h$ and $g$ since $\supp(xh^{-1}x^{-1})\subseteq T_{\alpha}^{x.w}$, so we further have that $[g,h]=[g,[h,x]]$. Therefore, 
\[
[g,h]=[g,[h,x]]=[\tilde{g},[h,x]]=[\tilde{g},[\tilde{h},x]], 
\]
and we deduce that $[g,h]\in A^{10}$.

The set $A^{10}$ thus contains every commutator of $\rist_G(w)$. In view of Lemma~\ref{lem:commutator width}, Theorem~\ref{thm:sji_char} implies that $k:=\cw(\rist_G(w))$ is finite, so $D(\rist_G(w))\leq A^{10k}$. 
\end{proof}

\subsection{The diagonalization trick} We now show we can find full sets under certain mild conditions. This lemma was already present in the work of J. Dixon, P. Neumann, and S. Thomas on the small index property for permutation groups \cite{zbMATH03981411}.

\begin{lem}\label{lem:diagonal argument}
Let $G\leq \Aut(T_{\alpha})$ be a closed subgroup and let $(A_n)_{n\in\N}$ be a countable family of subsets of $G$ such that $G=\bigcup_{n\in\N} A_n$. Then for any vertex $w\in T_{\alpha}$, there exists a vertex $v\geq w$ and $n\in\N$ so that $A_n$ is full above $v$. 
\end{lem}
\begin{proof}
Let $(w_n)_{n\in\N}$ enumerate the vertices of the rightmost branch of $T_{\alpha}^w$. For each $n$, let $v_n$ be a child of $w_n$ different from $w_{n+1}$. Let us prove by contradiction there is some $n\in\N$ such that $A_n$ is full above $v_n$; this implies the lemma.

If not, for each $n\in\N$ there is a tree automorphism $g_n\in G$ supported on $T_{\alpha}^{v_n}$ such that its restriction to $T_{\alpha}^{v_n}$ does not extend to an element of $A_n$. The products $\prod_{i=0}^ng_i$ converge to a limit $g$, and $g$ extends $g_i$ for all $i$. Since $G$ is closed in $\Aut(T_{\alpha})$, we have that $g\in G$, but for all $n\in\N$, the restriction of $g$ to $T^{v_n}_{\alpha}$ does not extend to an element of $A_n$. We conclude $g\not \in A_n$ for any $n$, contradicting the assertion $\bigcup_{n\in\N}A_n=G$. 
\end{proof}

Sets which are $\sigma$-syndetic are natural sources of sets full above a vertex.

\begin{lem}\label{lem: AA is full}
Let $G\leq \Aut(T_{\alpha})$ be a closed subgroup and let $A$ be a symmetric $\sigma$-syndetic subset of $G$. Then for any $w\in T_{\alpha}$, there exists $v\geq w$ such that $A^2$ is full above $v$. 
\end{lem}
\begin{proof}
We many find a sequence $(g_n)_{n\in\N}$ of elements of $G$ such that $\bigcup_{n\in\N}g_n A=G$. Lemma \ref{lem:diagonal argument} provides $n\in\N$ and a vertex $v\geq w$ such that $g_nA$ is full above $v$. Since $1\in \rist_G(v)$, there is $a\in g_nA$ so that $a\rest_{T_{\alpha}^v}$ is the identity. It now follows that $Ag^{-1}_ng_nA=A^2$ is full above $v$. 
\end{proof}

\section{Strongly just infinite profinite branch groups}

\subsection{Combinatorial boundedness conditions}
Our characterization of profinite branch groups with boundedness conditions requires a general observation.

\begin{lem}\label{lem:finality_commutator}
If $G$ is a profinite group with uncountable cofinality, then the derived group $D(O)$ is open for every open normal $O\normal G$. 
\end{lem}
\begin{proof}
We prove the contrapositive. Suppose that $O\normal G$ is open but that $D(O)$ is not. The subgroup $D(O)$ is Lebesgue measurable in $O$, hence the quotient group $O/D(O)$ is uncountable. Appealing to Lemma~\ref{lem:countable_quotient}, we may find $D(O)\leq A\normal O$ so that $O/A$ is countable and infinitely generated.

Taking $g_0,\dots,g_n$ left coset representatives for $O$ in $G$, the subgroup $\tilde{A}:=\bigcap_{i=0}^ng_iAg_i^{-1}$ is normal in $G$, and $G/\tilde{A}$ is countable. The group $G/\tilde{A}$ must be infinitely generated since $O/\tilde{A}$ is a finite index subgroup. Since infinitely generated countable groups plainly have countable cofinality, we conclude that $G$ has countable cofinality.
\end{proof}

\begin{thm}\label{thm:equiv_bergman}
Let $G\leq \Aut(T_\alpha)$ be a profinite branch group. Then the following are equivalent:
\begin{enumerate}[(1)]
\item $G$ is strongly just infinite.
\item $G$ has the Bergman property.
\item $G$ has uncountable cofinality.
\item $G$ has property (FA).
\item $G$ is Cayley bounded.
\end{enumerate}
\end{thm}

\begin{proof}
For $(1)\Rightarrow (2)$, let $(A_n)_{n\geq 0}$ be a Bergman sequence. Via Lemma~\ref{lem:diagonal argument}, there is $n_0\geq 0$ and $v\in T_{\alpha}$ so that $A_{n_0}$ is full above $v$. Appealing to Lemma~\ref{lem:comm_width}, we may find $w\geq v$ on some level $l$ so that $D(\rist_G(w))\subseteq A_{n_0}^{10k}$ for $k:=\cw(\rist_G(w))$.

Since $G$ acts transitively on the levels, there are $g_1,\dots,g_n\in G$ so that 
\[
\{g_1.w,\dots,g_n.w\}=V_l,
\]
and that $(A_n)_{n\geq 0}$ is a Bergman sequence ensures there is $n_1\geq n_0$ for which $g_1,\dots,g_n\in A_{n_1}$. It now follows
\[
\prod_{i=1}^ng_iA_{n_0}^{10k}g^{-1}_i\subseteq A_{n_1}^{10nk+2n}.
\]
We infer $D(\rist_G(l))\leq A_{n_1}^{10nk+2n}$. 

The set $A_{n_1}^{10nk+2n}$ contains the open subgroup $D(\rist_G(l))$ of $G$, which has finite index. Letting $h_1,\dots,h_m$ be left coset representatives for this subgroup, there is $n_2\geq n_1$ so that $h_1,\dots,h_m\in A_{n_2}$. We deduce that $A_{n_2}^{10nk+2n+1}=G$, hence $G$ has the Bergman property.

\

The implications $(2)\impl (3)$ and $(2)\impl(5)$ are given by Proposition~\ref{prop:bergman_cof,cay}. Theorem~\ref{thm:FA} ensures the equivalence $(3)\Leftrightarrow (4)$. It thus remains to show $(3)\impl (1)$ and $(5)\impl (1)$. The former is an easy exercise: The contrapositive follows from Lemma~\ref{lem:finality_commutator} and Theorem~\ref{thm:sji_char}.

To show $(5)\impl (1)$, we prove by contrapositive. Suppose $G$ is not strongly just infinite. In view of Theorem~\ref{thm:sji_char}, there is a level $k$ so that $D(\rist_G(k))$ is not open in $G$. The quotient $G/D(\rist_G(k))$ is then an infinite abelian-by-finite group. Applying Corollary~\ref{cor:ab_not_cay}, $G/D(\rist_G(k))$ is not Cayley bounded, whereby $G$ is not Cayley bounded.
\end{proof}

\subsection{Automatic continuity properties}
We now consider automatic continuity properties.
\begin{thm}\label{thm:equiv_1}
Suppose $G$ is a profinite branch group. Then the following are equivalent:
\begin{enumerate}[(1)]
\item $G$ is strongly just infinite.
\item $G$ has the weak Steinhaus property.
\item $G$ has the countable index property.
\item $G$ has the normal countable index property.
\end{enumerate}
\end{thm}

\begin{proof}
The implications $(2)\Rightarrow (3)$ and  $(3)\Rightarrow (4)$ are immediate. The contrapositive of $(4)\Rightarrow (1)$ follows from Lemma~\ref{lem:countable_quotient} and Theorem~\ref{thm:sji_char}.

For $(1)\Rightarrow (2)$, suppose $G$ is strongly just infinite and let $A\subseteq G$ be a $\sigma$-syndetic symmetric set. The set $\ol{A}$ is then also $\sigma$-syndetic. Applying the Baire category theorem, some left translate of $\ol{A}$ has non-empty interior, so $A$ is dense in some open set $V$. Since $V\inv V$ is a neighborhood of the identity and $A$ is symmetric, we deduce that $A^2$ is dense in a neighborhood of the identity. There is thus a level $n\in\N$ such that $A^2$ is dense in the pointwise stabilizer of $V_n$, denoted $\st_G(n)$. Fixing a symmetric set of right coset representatives $g_0,\dots,g_l$ for $\st_G(n)$ in $G$, put $B:=\bigcap_{i=0}^lg_i^{-1}A^{2^{l}}g_i$. Via Lemma~\ref{lem:sigma_sigma}, $B$ is again a symmetric $\sigma$-syndetic set.

For $v\in V_n$, we apply Lemma \ref{lem: AA is full} to find $w\geq v$ so that $B^2$ is full above $w$. By Lemma \ref{lem:sigma syndetic subgroup}, the set $B^4\cap \rist_G(w)$ is $\sigma$-syndetic in $\rist_G(w)$, and since $\rist_G(w)$ is uncountable, we have $B^4\cap\rist_G(w)\not\subseteq\{1\}$. Lemma~\ref{lem:comm_width} now implies we may find $s\geq w$ so that $D(\rist_G(s))\subseteq ({B^4})^{10k}=B^{40k}$.

Let $m$ be the level of $s$. The group $G$ acts transitively on $V_m$, so for all $t\in V_m$, there is $z\in G$ so that $z.s=t$. We may write $z=xg_i$ for some $x\in \st_G(n)$ and $g_i$ one of the previously fixed right coset representatives. Since $A^2$ is dense in $\st_G(n)$, there is $h\in A^2$ so that $hg_i. s=xg_i.s=t$.  We now have that $hg_iD(\rist_G(s))g_i^{-1}h^{-1}=D(\rist_G(t))$. Moreover, 
\[
hg_iD(\rist_G(s))g_i^{-1}h^{-1}\subseteq hg_iB^{40k}g_i^{-1}h^{-1}\subseteq A^{40k2^{l}+4}.
\]
We conclude that $A^{40k2^{l}+4}$ contains $D(\rist_G(t))$ for all $t\in V_m$. The open subgroup $D(\rist_G(m))$ is thus contained in $A^{(40k2^{l}+4)|V_m|}$, whereby $G$ enjoys the weak Steinhaus property.
\end{proof}


We pause for a moment to observe that the branch assumption in Theorem~\ref{thm:equiv_1} is necessary. 
\begin{prop}\label{prop: pslnZp fails SIP}
For $n\geq 2$, the profinite group $PSL_n(\Z_p)$ fails the countable index property but is strongly just infinite.
\end{prop}
\begin{proof} The profinite group $PSL_n(\Z_p)$ is strongly just infinite via the main theorem of the appendix of \cite{R70}. 

On the other hand, via \cite[Theorem 1]{K00}, there is an injective homomorphism $\xi:GL_n(\C)\rightarrow \Sym_{\infty}$, and since the algebraic closure of $\Q_p$ is isomorphic to $\C$, we may see $SL_n(\Q_p)\leq GL_n(\C)$. For each $\alpha \in \Aut(\C)$, the map $\phi_{\alpha}$ given by applying $\alpha$ to the entries of a matrix is an automorphism of $GL_n(\C)$. We therefore obtain maps $\xi\circ \phi_{\alpha}:SL_n(\Q_p)\rightarrow \Sym_{\infty}$ for each $\alpha \in \Aut(\C)$. 

For $\alpha$ and $\beta$ in $\Aut(\C)$, the maps $\xi\circ \phi_{\alpha}$ and $\xi\circ \phi_{\beta}$ are equal if and only if 
\[
\phi_{\alpha}\rest_{SL_n(\Q_p)}=\phi_{\beta}\rest_{SL_n(\Q_p)}.
\]
The maps $\phi_{\alpha}$ and $\phi_{\beta}$ agree on $SL_n(\Q_p)$ if and only if $\phi_{\alpha^{-1}\circ \beta}$ is the identity on $SL_{n}(\Q_p)$. The group $SL_n(\Q_p)$ contains elementary matrices $E_{i,j}(a)$ for $i\neq j$ and $a\in \Q_p$ where $E_{i,j}(a)$ has ones on the diagonal, $a$ in the $(i,j)$-entry, and zeros elsewhere. We conclude $\alpha^{-1}\circ \beta (a)=a$, so $\alpha^{-1}\circ \beta$ fixes  $\Q_p$ pointwise. Therefore, $\xi\circ \phi_{\alpha}$ and $\xi\circ \phi_{\beta}$ are equal if and only if $\alpha^{-1}\circ \beta\in \Aut(\C/\Q_p)$. 

It is well-known $|\Aut(\C)|=2^{\mathfrak{c}}$. On the other hand, $\Aut(\C/\Q_p)$ is a second countable profinite group and thus has size $\mathfrak{c}$. We conclude there are $2^{\mathfrak{c}}$ many distinct left cosets of $\Aut(\C/\Q_p)$ in $\Aut(\C)$. In view of the previous paragraph, there must be $2^{\mathfrak{c}}$ many distinct homomorphisms $\xi\circ\phi_{\alpha}:SL_n(\Q_p)\rightarrow \Sym_{\infty}$.  Since there can be at most continuum many continuous homomorphisms, we conclude that $SL_n(\Q_p)$ fails the countable index property, and as $SL_n(\Z_p)$ is an open subgroup of $SL_n(\Q_p)$, the group $SL_n(\Z_p)$ also fails the countable index property. It now follows $PSL_n(\Z_p)$ fails the countable index property.
\end{proof}

\subsection{Commensurated subgroups}

We finally give a characterization of  strongly just infinite profinite branch groups in terms of commensurated subgroups. This classification follows from a general observation about the Bergman property.

\begin{prop}\label{prop:bergman_comm}
If a Polish group $G$ has the Bergman property, then every commensurated subgroup is commensurate to a normal subgroup.
\end{prop}
\begin{proof}
Suppose $C\leq G$ is commensurated. For each $n\geq 1$, set
\[
\Omega_n:=\{g\in G\mid |C:C\cap gCg^{-1}|\leq n \text{ and } |gCg^{-1}:C\cap gCg^{-1}|\leq n\}
\]
The sets $\Omega_n$ are symmetric, and since $C$ is commensurated, $G=\bigcup_{n\geq 1}\Omega_n$. For all $n,m\geq 1$, an easy computation further verifies $\Omega_n\Omega_m\subseteq \Omega_{nm}$.  
Since $G$ has the Bergman property, there is $n,k$ so that $\Omega_{n}^k=G$, whereby $\Omega_{n^k}=G$. Appealing to Theorem~\ref{thm:BL}, there is $L\normal G$ which is commensurate with $C$, verifying the proposition.
\end{proof}

\begin{thm}\label{thm:comm_char_sji}
Suppose $G\leq \Aut(T_{\alpha})$ is a profinite branch group. Then $G$ is strongly just infinite if and only if every commensurated subgroup of $G$ is either finite or open.\end{thm}

\begin{proof}
Suppose that $G$ is strongly just infinite and that $C\leq G$ is a commensurated subgroup. Since $G$ is strongly just infinite, $G$ has the Bergman property via Theorem~\ref{thm:equiv_bergman}. Proposition~\ref{prop:bergman_comm} then supplies $D\normal G$ so that $C\sim_c D$, and since $G$ is strongly just infinite, $D$ is either open or trivial. If $D$ is trivial, then $C$ is finite. If $D$ is open, then $C$ has finite index in $G$. That $G$ is strongly just infinite implies the normal core of $C$ in $G$ is open, whereby $C$ is open.

\

Conversely, suppose every commensurated subgroup of $G$ is either finite or open. Fix a level $n$ of $T_{\alpha}$. The subgroup $D(\rist_G(n))$ is normal in $G$, so a fortiori, it is commensurated. Suppose toward a contradiction that $D(\rist_G(n))$ is finite. The subgroups $\st_G(k)$ form a basis at $1$ for $G$, so we may find $m\geq n$ for which $\st_G(m)\cap D(\rist_G(n))=\{1\}$ and $\st_G(m)\leq \rist_G(n)$.

The group $\st_G(m)$ thus injects into $\rist_G(n)/D(\rist_G(n))$ and therefore is abelian. For any $v\in V_m$, the rigid stabilizer $\rist_G(v)$ then has a non-trivial center, but this is absurd in view of Lemma~\ref{lem:no_centre}. We conclude $D(\rist_G(n))$ is open for all levels $n$. Theorem~\ref{thm:sji_char} now implies $G$ is strongly just infinite.
\end{proof}

Bringing together Theorems~\ref{thm:equiv_bergman}, \ref{thm:equiv_1}, and \ref{thm:comm_char_sji}, we have established the claimed equivalences.

\begin{thm}\label{thm:equiv}
Let $G\leq \Aut(T_\alpha)$ be a profinite branch group. Then the following are equivalent:
\begin{enumerate}[(1)]
\item $G$ is strongly just infinite.
\item Every commensurated subgroup of $G$ is either finite or open.
\item $G$ has the normal countable index property.
\item $G$ has the countable index property.
\item $G$ has the weak Steinhaus property.
\item $G$ has uncountable cofinality.
\item $G$ is Cayley bounded.
\item $G$ has property (FA).
\item $G$ has the Bergman property.
\end{enumerate}
\end{thm}

\section{Invariant and locally compact automatic continuity properties}\label{sec: inv AC and lc AC}
We now investigate the invariant automatic continuity property and the locally compact automatic continuity property in profinite branch groups. For these results, we restrict our class slightly.

For a finite permutation group $(F,\Omega)$, a \textbf{derangement} of $\Omega$ is a permutation $f\in F$ so that $f$ fixes no point in $\Omega$. It is an easy, amusing exercise to see that every finite transitive permutation group $(F,\Omega)$ with $|\Omega|>1$ contains a derangement; this observation is originally due to C. Jordan. Given a derangement $f\in F$, we may write $f$ as a product of disjoint cycles $f=c_1\dots c_n$, and each $c_i$ has length at least $2$. 

Derangements appear in our discussion by way of the following definition:

\begin{df} 
We say a profinite branch group $G\leq \Aut(T_{\alpha})$ \textbf{locally has derangements} if for each $n\geq 0$ there is $N\geq n $ for which $\st_G(n)$ contains a derangement of $V_N$.
\end{df}
\noindent Groups built by iterated wreath products of transitive permutation groups are easy examples of branch groups which locally have derangements. Indeed, suppose $G\leq \Aut(T_{\alpha})$ is such an iterated wreath product. For each $v\in V_n$, the rigid stabilizer $\rist_G(v)$ acts transitively on the children of $v$ in $V_{n+1}$, so there is $x_v\in \rist_G(v)$ a derangement of the children of $v$ in $V_{n+1}$. The element $\prod_{v\in V_n}x_v$ is then an element of $\rist_G(n)$ which is a derangement of $V_{n+1}$. The reader is encouraged to look ahead to Section~\ref{sec:examples} to see examples of such constructions.

We remark that we do not know of a profinite branch group which fails to locally have derangements. As this seems an independently interesting question, we set it out explicitly:
\begin{qu}
Does every (profinite) branch group locally have derangements? 
\end{qu}

We shall also need to have control over commutator widths of the rigid stabilizers.
\begin{df}
A profinite branch group $G\in \Aut(T_{\alpha})$ is said to have \textbf{uniform commutator widths} if  $\sup\{\cw(\rist_G(v))\mid v\in T_{\alpha}\}= c<\infty$. The value $c$ is called a \textbf{uniform bound} for the commutator widths.
\end{df}
\noindent Examples of profinite branch groups with uniform commutator widths are also presented in Section~\ref{sec:examples}.

\subsection{Automatic continuity results}

\begin{prop}\label{prop:invariant_steinhaus_wreath}
Suppose $G$ is a profinite branch group which locally has derangements and has uniform commutator widths with uniform bound $c$. If $A\subseteq G$ is an invariant symmetric $\sigma$-syndetic set, then $D(\rist_G(k))\leq A^{24c}$ for some level $k$.
\end{prop}

\begin{proof}
Since $A$ is $\sigma$-syndetic, the Baire category theorem implies that $A$ is dense in some open set $V$. The set $A^2$ is then dense in the neighborhood of the identity $V\inv V$. Let $n$ be so that $\st_G(n)$ is contained in $\ol{A^2}$. 

As $G$ locally has derangements, there is $k\geq n$ and $y\in \st_G(n)$ so that $y$ is a derangement of $V_{k}$. The set $A^2$ is dense in $\st_G(n)$, whereby we may find $z\in A^2$ so that $z.w=y.w$ for all $w\in V_{k}$. Hence, $z$ acts as a derangement on $V_{k}$.

Let the action of $z$ on $V_{k}$ be given by the product of disjoint cycles $c_1\dots c_m$. This action is a derangement, so each cycle has length at least $2$. A cycle $c$ may be written as a tuple $(w_{i_0},\dots,w_{i_{l-1}})$ of vertices from $V_{k}$ so that $c:w_{i_{j}}\rightarrow w_{i_{j+1\mod l}}$. We may thus choose every second vertex appearing in $c$; that is to say, we take $w_{i_1},w_{i_3},\dots$. Let $Z$ list every second vertex from each of the $c_1,\dots, c_m$.

Consider the subgroup $H:=\prod_{v\in Z}\rist_G(v)$. For each $h\in H$, we see that $(z.\supp(h))\cap \supp(h)=\emptyset$. Lemma~\ref{lem:comm_trick} therefore implies that every commutator $[g,t]$ with $g,t\in H$ is a product of four conjugates of $z^{\pm 1}$. Recalling $A$ is conjugation invariant, we conclude that $[H,H]\subseteq A^{8}$. Since the derived subgroup of $H$ is $\prod_{v\in Z}D(\rist_G(v))$, the group $H$ has commutator width at most $c$, hence $D(H)\leq A^{8c}$.  

Since $z$ is a derangement, it follows that $z^{-1}Zz\cup Z\cup zZz^{-1}=V_{k}$, and as $A$ is conjugation invariant, we infer that
\[
\prod_{w\in zZz^{-1}} D(\rist_G(w))\cup \prod_{w\in z^{-1}Zz}D(\rist_G(w))\subseteq A^{8c}.
\]  
Hence, $D(\rist_G(k))\leq A^{24c}$, verifying the proposition.
\end{proof}

\begin{thm}\label{thm:invariant_ACP}
If $G$ is a strongly just infinite profinite branch group which locally has derangements and has uniform commutator widths, then $G$ has the invariant automatic continuity property.
\end{thm}
\begin{proof}
In view of Proposition~\ref{prop:invariant_steinhaus}, it suffices to show that $G$ satisfies the invariant Steinhaus property. 

Let $c>0$ be a uniform bound on the commutator widths and suppose $A\subseteq G$ is an invariant $\sigma$-syndetic subset of $G$. Proposition~\ref{prop:invariant_steinhaus} ensures $A^{24c}$ contains $D(\rist_G(n))$ for some level $n$. The group $G$ is strongly just infinite, so $D(\rist_G(n))$ is open via Theorem~\ref{thm:sji_char}. The subset $A^{24c}$ therefore contains a neighborhood of $1$, whereby $G$ has the invariant Steinhaus property with constant $24c$. 
\end{proof}

\begin{cor}\label{cor:lc_continuity}
If $G$ is a strongly just infinite profinite branch group which locally has derangements and has uniform commutator widths, then $G$ has the locally compact automatic continuity property.
\end{cor}
\begin{proof}
By Lemma \ref{lem: iac+bergman implies lcac} we need only to check that $G$ has the Bergman property and the invariant automatic continuity property, and these are given by Theorems \ref{thm:equiv_bergman} and \ref{thm:invariant_ACP}.\end{proof}

The conditions that $G$ is strongly just infinite and has uniform commutator widths are not redundant. In a private communication, N. Nikolov explained to us an example of a profinite branch group $G$ which is strongly just infinite and locally has derangements, but fails to have uniform commutator widths.

The examples given by \cite[Theorem 1]{W06} show uniform commutator widths and local derangements do not imply strongly just infinite. In these examples, all rigid stabilizers are topologically two generated, so the celebrated solution to Serre's question by Nikolov and D. Segal, cf. \cite{NS12}, implies there is a uniform bound on the commutator widths. They locally have derangements as they are built via iterated wreath products. On the other hand, they surject onto an infinite profinite abelian group and so are not strongly just infinite.

\section{Rigidity of the group topology}\label{sec:unique topology}
We now consider the group topologies a profinite branch group admits.
\begin{thm}\label{thm:unique_topology}
The profinite topology of a profinite branch group is its unique Polish group topology as well as its unique compact Hausdorff group topology.
\end{thm}

\begin{proof}
Suppose $G\leq \Aut(T_{\alpha})$ is a profinite branch group and suppose $\psi:G\rightarrow H$ is a bijective homomorphism with $H$ a topological group. For each $v\in T_{\alpha}$ with $v\in V_n$, Corollary~\ref{lem:description_rist} gives that
\[
\psi(\rist_G(v))=\bigcap_{w\in V_n\setminus\{v\}}\psi(C_G(\rist_G(w)))=\bigcap_{w\in V_n\setminus\{v\}}C_H(\psi(\rist_G(w))).
\]
As centralizers are always closed, we conclude that $\psi(\rist_G(v))$ is closed in $H$.

If $H$ is a Polish group, we deduce that the subgroup $\psi(\rist_G(n))$ is analytic, as it is a finite product of closed sets, so it is Baire measurable for all levels $n$. The subgroup $\psi(\rist_G(n))$ is also finite index in $H$, so via the Baire category theorem, $\psi(\rist_G(n))$ is indeed open. We deduce that the map $\psi^{-1}:H\rightarrow G$ is continuous, whereby $\psi$ is continuous since both $G$ and $H$ are Polish. It now follows that $G$ has a unique Polish group topology.

If $H$ is a compact group, for every $n\in\N$ the subgroup $\psi(\rist_G(n))$ is compact as the product of finitely many compact sets, so it is closed. Since $\psi(\rist_G(n))$ has finite index, it must be open. We conclude that the map $\psi^{-1}:H\rightarrow G$ is continuous, so by compactness, $\varphi$ is a homeomorphism. Hence, $G$ has a unique compact group topology.
\end{proof}

Under the additional assumption of being strongly just infinite, we can upgrade our rigidity results.
\begin{thm} 
A strongly just infinite profinite branch group admits exactly two locally compact Hausdorff group topologies: the discrete topology and the profinite topology of a profinite branch group.
\end{thm}
\begin{proof}
Let $G\leq \Aut(T_{\alpha})$ be a profinite branch group and suppose $\psi:G\rightarrow H$ is a bijective homomorphism with $H$ a locally compact group. Consider first the connected component $H^{\circ}\normal H$. Since $G$ is strongly just infinite $\psi^{-1}(H^{\circ})$ is either trivial or open with finite index. 

Let us eliminate the latter case first. In this case, $H$ is almost connected, and since connected locally compact groups are compactly generated, $H$ is compactly generated. Theorem \ref{thm:equiv_bergman} now implies $H$ is compact. This is absurd since $G$ has a unique compact group topology by Theorem \ref{thm:unique_topology}.

It is therefore the case that $H^{\circ}$ is trivial, so $H$ is a totally disconnected locally compact group. Assume that $H$ is non discrete. By van Dantzig's theorem, $H$ admits a basis at $1$ of infinite compact open subgroups. Let $U$ be such a subgroup. Since compact open subgroups are necessarily commensurated,  $\psi^{-1}(U)$ is a commensurated subgroup of $G$, whereby Theorem~\ref{thm:comm_char_sji} implies $\psi^{-1}(U)$ is open in $G$. It now follows the map $\psi$ is continuous, hence $G\simeq H$ as topological groups.
\end{proof}

\section{Examples}\label{sec:examples}
Let $(A,X)$ and $(B,Y)$ be finite permutation groups. We may form the group $B\wr (A,X):=B^{X}\rtimes A$ where $A\act B^X$ by permuting the domain. The group $B^{X}\rtimes A$ is a permutation group via the following canonical action on $X\times Y$:
\[
(f,a).(x,y):=(a.x,f(a.x).y).
\]
The \textbf{wreath product} of $(B,Y)$ with $(A,X)$, denoted $(B,Y)\wr(A,X)$, is the permutation group $(B\wr (A,X),X\times Y)$. Wreath products defined in this way are associative. 

For an infinite sequence $((A_i,X_i))_{i\in \N}$ of finite permutation groups, the set of finite wreath products $(A_n,X_i)\wr\dots \wr(A_0,X_0)$ forms an inverse system via the obvious quotient maps 
\[
(A_{n+1},X_{n+1})\wr\dots \wr(A_0,X_0)\rightarrow (A_n,X_n)\wr\dots \wr(A_0,X_0).
\]
We define
\[
W((A_i,X_i)_{i\in \N}):=\ilim_{n\in \N}\left((A_n,X_n)\wr\dots \wr(A_0,X_0)\right).
\]

The action of the finite wreath products $(A_n,X_n)\wr\dots \wr(A_0,X_0)$ on the product $X_0\times \dots \times X_n$ induces an action of the group $W((A_i,X_i)_{i\in \N})$ on the tree $T_{\alpha}$ where $\alpha(i):=|X_i|$. When the permutation groups $(A_n,X_n)$ are transitive and non-trivial for all $n$, the action of the group $W((A_i,X_i)_{i\in \N})$ on $T_{\alpha}$ witnesses that $W((A_i,X_i)_{i\in \N})$ is a profinite branch group. The rigid stabilizers are also easy to understand: If $v\in T_{\alpha}$ lies on level $n$, then
\[
\rist_G(v)\simeq W((A_i,X_i)_{i>n}).
\]

\subsection{Iterated wreath products}
Many of the groups $W((A_i,X_i)_{i\in \N})$ are strongly just infinite, locally have derangements, and have uniform commutator widths. The latter requires a theorem due to N. Nikolov.

\begin{thm}[Nikolov, {\cite[Corollary 1.4]{N04}}]\label{thm:Nikolov}
Suppose $((A_i,X_i))_{i\in \N}$ is a sequence of finite perfect permutation groups.  If $\sup\{\cw(A_i)\mid i\in \N\}=N<\infty$, then $W((A_i,X_i)_{i\in \N})$ is perfect as an abstract group, and $\cw(W((A_i,X_i)_{i\in \N}))\leq N$. 
\end{thm}

Via Theorem~\ref{thm:Nikolov}, we isolate a rich family of profinite branch groups to which our results apply.

\begin{prop}\label{prop:examples}
Suppose $((A_i,X_i))_{i\in \N}$ is a sequence of finite non-trivial perfect transitive permutation groups. If $\sup\{\cw(A_i)\mid i\in \N\}<\infty$, then $W((A_i,X_i)_{i\in \N})$ is strongly just infinite, locally has derangements, and has uniform commutator widths. 
\end{prop}
\begin{proof}
Suppose $\sup\{\cw(A_i)\mid i\in \N\}=N$, set $G:=W((A_i,X_i)_{i\in \N})$, and let $T_{\alpha}$ be the rooted tree on which $G$ acts, as discussed above. For each $v\in T_{\alpha}$ with $v\in V_n$, we have that $\rist_G(v)\simeq W((A_i,X_i)_{i>n})$, so Theorem~\ref{thm:Nikolov} implies $\rist_G(v)$ is abstractly perfect and $\cw(\rist_G(v))\leq N$. The group $G$ thus has uniform commutator widths, and Theorem~\ref{thm:sji_char} ensures $G$ is strongly just infinite. That $G$ locally has derangements follows since $\rist_G(v)\simeq W((A_i,X_i)_{i>n})$.
\end{proof}

\begin{rmq}We note a weak converse: if infinitely many of the $A_i$ are not perfect, then $W((A_i,X_i)_{i\in \N})$ fails to be just infinite. Indeed, if $((A_i,X_i))_{i\in \N}$ is any sequence of finite transitive permutation groups, the group $W((A_i,X_i)_{i\in \N})$ surjects continuously onto the abelian group $\prod_{i\in\N}A_i/D(A_i)$.

An easy example of this situation is provided by the full automorphism group of $T_\alpha$. The group $\Aut(T_{\alpha})$ may be written as $W((\mathfrak S(\alpha_i), [\alpha_i])_{i\in\N})$ and thus surjects onto $\prod_{i\in\N}\Z/2\Z$.
\end{rmq}

\begin{cor}
Suppose $((A_i,X_i))_{i\in \N}$ is any sequence of non-abelian finite simple transitive permutation groups. Then, $W((A_i,X_i)_{i\in \N})$ is strongly just infinite, locally has derangements, and has uniform commutator widths. 
\end{cor}
\begin{proof}
Non-abelian finite simple groups have commutator width one by the celebrated solution to the Ore conjecture \cite{LOST10}. The desired result then follows from Proposition~\ref{prop:examples}.
\end{proof}

\subsection{The profinite completion of the first Grigorchuk group}
Our discussion of the first Grigorchuk group follows V. Nekrashevych's work \cite{N05}. The first Grigorchuk group, denoted $G_1$, is a four generated group that acts on the two regular rooted tree $T_2$. Identifying the vertices of $T_2$ with finite binary sequences in the obvious way, the generators $a,b,c,d$ of $G_1$ are defined recursively as follows:
\[
\begin{array}{cc}
a.(0\conc \alpha):=1\conc \alpha & a.(1\conc \alpha):=0\conc \alpha\\
b.(0\conc \alpha):=0\conc a.\alpha & b.(1\conc \alpha):=1\conc c.\alpha\\
c.(0\conc \alpha):=0\conc a.\alpha & c.(1\conc \alpha):=1\conc d.\alpha\\
d.(0\conc \alpha):=0\conc\alpha & d.(1\conc \alpha):=1\conc b.\alpha
\end{array}
\]
Letting $\sigma$ be the non-trivial element of $\Z/2\Z$, these generators can be given by so-called wreath-recursion as follows: $a:=(1,1)\sigma$, $b:=(a,c)$, $c:=(a,d)$, and $d:=(1,b)$. These are forms of the generators upon realizing $\Aut(T_2)$ as $\Aut(T_2)^2\rtimes \Z/2\Z$.

\begin{fact}[Grigorchuk]\label{fact:grigorchuk_group}
The first Grigorchuk group $G_1$ enjoys the following properties:
\begin{enumerate}[(1)]
\item It is an infinite two group. \textup{(See \cite[Theorem 1.6.1]{N05}.)}
\item It is a branch group. \textup{(See \cite[Corollary of Proposition 8]{G00}.)}
\item It is just infinite. \textup{(See \cite[Corollary of Proposition 9]{G00}.)}
\item It has the congruence subgroup property. \textup{(See \cite[Proposition 10]{G00}.)}
\end{enumerate}
\end{fact} 

\begin{lem}\label{lem:degrangement}
For every $n\geq 0$, the rigid level stabilizer $\rist_{G_1}(n)$ contains a derangement of $V_{n+2}$.
\end{lem}
\begin{proof}
Via \cite[Proposition 8]{G00}, the subgroup $K:=\grp{x,(x,1),(1,x)}\leq G_1$ with $x:=(ca,ac)$ is self-replicating. The element $x$ obviously acts on as a derangement on the level two vertices of $T_2$. 

For an arbitrary level $n\geq 0$, that $K$ is self-replicating ensures that $K\leq \rist_K(v)\leq \rist_{G_1}(v)$ for all $v\in V_n$. We may thus find $x_v\in K=\rist_K(v)$ so that $x_v$ is a derangement of level $2$ of the tree $T_2^v$. The element $z:=\prod_{v\in V_m}x_v\in \rist_{G_1}(n)$ is then a derangement of $V_{n+2}$, verifying the lemma.
\end{proof}


\begin{lem}\label{lem:comm}
For every $v\in T_2$, the rigid stabilizer $\rist_{G_1}(v)$ is at most $19$ generated.
\end{lem}
\begin{proof}
The group $G_1$ is self-similar since it is given by a finite automata; cf. \cite[Section 1.5]{N05}. For each $v\in T_1$, identifying $T_2^v$ with $T_2$ thus gives an injective homomorphism $\psi:\rist_{G_1}(v)\rightarrow G_1$. 

Take the subgroup $K:=\grp{x,(x,1),(1,x)}$ for $x:=(ca,ac)$. As noted above, $K$ is self-replicating, so for each vertex $v\in T_2$, we have $K\leq \rist_K(v)$. We claim that indeed $K=\rist_K(v)$. Since $K$ is self-replicating, it is enough to show $\rist_K(0)=K$, where $0$ is the vertex of $T_2$ labeled by the binary sequence $0$. By \cite[Proposition 8]{G00}, the quotient $K/K_1$ with $K_1:=K\times K$ is a cyclic group of order four generated by $(ab)^2=:y$. 

We now suppose $y^i(n,m)\in \rist_K(0)$ for some $(n,m)\in K_1$ and $1\leq i \leq 4$. One can verify the element $y^i$ fixes the vertex $1$, so it must be the case that $m=1$. Therefore, $y^i(n,1)\in \rist_K(0)$, and since $(n,1)\in \rist_K(0)$, we conclude that $y^i\in \rist_K(0)$. It now follows that $i=4$, and therefore, $K=\rist_K(0)$.

The map $\psi$ thus induces an injection $\rist_{G_1}(v)/K\rightarrow G_1/K$.  Since $G_1/K$ has order $16$ via \cite[Proposition 8]{G00}, we conclude that $\rist_{G_1}(v)/K$ has order at most $16$, so $\rist_{G_1}(v)$ is at most $19$ generated.
\end{proof}

We are now ready to prove the desired result.

\begin{thm} 
The profinite completion of the first Grigorchuk group is a profinite branch group which is strongly just infinite, locally has derangements, and has uniform commutator widths.
\end{thm}
\begin{proof}
Since $G_1$ has the congruence subgroup property, \cite[Theorem 9]{G00} implies the profinite completion $\widehat{G}_1$ is isomorphic to the topological closure $\ol{G_1}$ in $\Aut(T_2)$. For each level $n$, the rigid level stabilizer $\rist_{G_1}(n)$ additionally contains $\st_{G_1}(m)$ for some $m\geq n$ via the congruence subgroup property. The subgroup $\ol{\rist_{G_1}(n)}$ is then open in $\ol{G_1}$, so $\ol{G_1}$ is a profinite branch group. It follows further that $\rist_{\ol{G_1}}(v)=\ol{\rist_{G_1}(v)}$.

Lemma~\ref{lem:degrangement} now implies that $\ol{G_1}$ locally has derangements. Lemma~\ref{lem:comm} ensures the rigid stabilizer $\rist_{\ol{G_1}}(v)$ is topologically $19$ generated for any $v\in T_2$, and since $\ol{G_1}$ is a pro-$2$ group, $\rist_{\ol{G_1}}(v)$ has commutator width at most $19$ by classical results - e.g. \cite[Proposition 4.7]{CER08}. The group $\ol{G_1}$ therefore also has uniform commutator widths. 

Finally, that $\rist_{\ol{G_1}}(v)$ has finite commutator widths implies $D(\rist_{\ol{G_1}}(v))$ is closed. On the other hand, $\rist_{G_1}(v)$ is dense in $\rist_{\ol{G_1}}(v)$, and $D(\rist_{G_1}(v))$ has finite index in $\rist_{G_1}(v)$, since $G_1$ is just infinite. It follows $D(\rist_{\ol{G_1}}(v))$ has finite index in $\rist_{\ol{G}_1}(v)$. Theorem~\ref{thm:sji_char} now implies $\ol{G_1}$ is strongly just infinite. 
\end{proof}

\section{Application: Burger-Mozes universal groups}
Let $\Trd$ be the $d$-regular \textit{unrooted} tree for $d\geq 3$. Let $V\Trd$ and $E\Trd$ denote the collection of vertices and edges. For a vertex $v\in V\Trd$, the set $E(v)$ denotes the collection of edges containing $v$. Recall $[d]:=\{0,1,\dots, d-1\}$ and fix a map $c:E\Trd\rightarrow [d]$ so that $c\rest_{E(v)}=:c_v$ is a bijection for each $v\in V$; the map $c$ is called a \textbf{coloring}. For each $v\in V\Trd$ and $g\in \Aut(\Trd)$, there is a bijection $g_v:E(v)\rightarrow E(g.v)$ arising from the action of $g$. 

\begin{df}
Let $d\geq 3$, $F\leq \Sym_d$, and $c:E\Trd\rightarrow [d]$ be a coloring. The \textbf{Burger--Mozes universal group} is defined to be
\[
U(F):=\{g\in \Aut(\Trd)\mid c_{g.v}\circ g_v\circ c_v^{-1}\in F\text{ for all }v\in V\Trd\}.
\]
The subgroup generated by all pointwise edge stabilizers in $U(F)$ is denoted $U(F)^+$.
\end{df} 
The group $U(F)$ is closed in $\Aut(\Trd)$ and so is a totally disconnected locally compact Polish group. The isomorphism type of the group $U(F)$ is independent of the coloring $c$; indeed, any two groups built using different colorings are conjugate in $\Aut(\Trd)$. We therefore will suppress the coloring.

\begin{prop}[Burger--Mozes, {\cite[Proposition 3.2.1]{BM00}}]
Let $d\geq 3$ and let $F\leq \Sym_d$ be transitive and generated by its point stabilizers. Then, $U(F)^+$ is a non-discrete compactly generated totally disconnected locally compact Polish group which is simple and has index two in $U(F)$.
\end{prop}
The group $U(F)^+$ is often called the \textbf{Burger--Mozes universal simple group}. The structure of a compact open subgroup of $U(F)^+$ is well-understood.

\begin{prop}[Burger--Mozes, {\cite[Section 3.2]{BM00}} (cf. {\cite[Proposition 4.3]{CdM11}})]\label{prop:compact_open_BM_grp}
Let $d\geq 3$, let $F\leq \Sym_d$ be transitive and generated by its point stabilizers, and set $H:=\mathrm{Stab}_F(v)$ for some $v\in [d]$. The group $U(F)^+$ has a compact open subgroup isomorphic to $W((A_i,X_i)_{i\in \N})$ where
\[
(A_i,X):=\begin{cases} (F,[d]) &\mbox{if } i= 0 \\ 
(H,[d]\setminus\{v\}) & \mbox{else} . 
\end{cases}
\]  
\end{prop}

We are now ready to apply our results on branch groups. Since $\Sym_n$ is solvable for $n\leq 4$, we shall have to consider $\Trd$ with $d\geq 6$.

\begin{thm}\label{thm:BM_aut} 
Suppose that $d\geq 6$ and that $F\leq \Sym_d$ is perfect, two transitive, and generated by point stabilizers. Suppose further the point stabilizers of $F$ are also perfect. The group $U(F)^+$ then enjoys the countable index property, the invariant automatic continuity property, and the locally compact automatic continuity property. 
\end{thm}
\begin{proof}
Let $H:=\mathrm{Stab}_F(v)$ for some $v\in [d]$. Proposition~\ref{prop:compact_open_BM_grp} ensures that $U(F)^+$ has a compact open subgroup $U$ isomorphic to $W((A_i,X_i))_{i\in \N}$ where
\[
(A_i,X):=\begin{cases} (F,[d]) &\mbox{if } i= 0 \\ 
(H,[d]\setminus\{v\}) & \mbox{else} . 
\end{cases}
\] 

The groups $(F,[d])$ and $(H,[d]\setminus\{v\})$ are perfect transitive permutation groups, so Proposition~\ref{prop:examples} implies the compact open subgroup $U$ is strongly just infinite, locally has derangements, and has uniform commutator widths. Theorems~\ref{thm:equiv} and \ref{thm:invariant_ACP} and Corollary~\ref{cor:lc_continuity} imply $U$ has the countable index property, the invariant automatic continuity property, and the locally compact automatic continuity property.

Let $K$ be either a non-archimedean Polish group, a SIN Polish group, or a locally compact Polish group and suppose $\psi:U(F)^+\rightarrow K$ is a homomorphism. Taking $O\subseteq K$, the restriction $\psi\rest_U:U\rightarrow K$ must be continuous, so $(\psi\rest_U)^{-1}(O)$ contains an open subset of $U$. Therefore, $\psi^{-1}(O)$ contains an open subset of $U(F)^+$, and it follows that $\psi$ is continuous.
\end{proof}

\begin{cor} For $A_n\leq \Sym_n$ the alternating group with $n\geq 6$, the Burger-Mozes universal simple group $U(A_n)^+$ has the countable index  property, the invariant automatic continuity property, and the locally compact automatic continuity property. 
\end{cor}
\begin{proof} 
The group $A_n$ is perfect, two transitive on $[n]$, and generated by its point stabilizers. The point stabilizers are $A_{n-1}$, so they are also perfect.  Theorem~\ref{thm:BM_aut} now implies the corollary.
\end{proof} 


Of course, that the Burger--Mozes universal simple groups have these automatic continuity properties follow from our results on profinite groups. One naturally asks if such examples can be found using Theorem~\ref{thm:products} instead. It turns out this is not possible: \textit{There is no compactly generated locally compact group which is topologically simple and contains an infinite product of non-trivial finite groups as a compact open subgroup}. This follows by considering the quasi-center of such a group and applying \cite[Proposition 4.3]{CM11}.


We conclude by classifying the commensurated subgroups of certain Burger--Mozes universal simple groups.

\begin{thm}\label{thm:BM_comm}
Suppose that $d\geq 6$ and that $F\leq \Sym_d$ is perfect, two transitive, and generated by point stabilizers. Suppose further the point stabilizers of $F$ are also perfect. Then every commensurated subgroup of $U(F)^+$ is ether finite, compact and open, or equal to $U(F)^+$. 
\end{thm}

\begin{proof}
Let $U$ be the compact open subgroup given by Proposition~\ref{prop:compact_open_BM_grp} and let $O\leq G:=U(F)^+$ be a commensurated subgroup of $G$. The group $O\cap U$ is then a commensurated subgroup of $U$, and since $U$ is strongly just infinite, Theorem~\ref{thm:equiv} implies $O\cap U$ is either open or finite. If $O\cap U$ is open, then $O$ is open. Via \cite[Proposition 4.1]{CdM11}, $O$ is either compact and open or equal to $G$; in either case we are done.

We thus suppose that $O\cap U$ is finite. If $O$ is finite, we are done, so we suppose for contradiction that $O$ is infinite. The sets 
\[
\Omega_n:=\{u\in U\mid |O:O\cap uOu^{-1}|\leq n \text{ and }|uOu^{-1}:O\cap uOu^{-1}|\leq n\}
\]
form an increasing exhaustion of $U$ by symmetric sets since $O$ is commensurated; recall for all $n,m\geq 1$, we have $\Omega_n\Omega_m\subseteq \Omega_{nm}$. The group $U$ has the Bergman property via Theorem~\ref{thm:equiv}, whereby $\Omega_m=U$ for some sufficiently large $m$.  Appealing to Theorem~\ref{thm:BL}, we conclude there is $O'\sim_c O$ so that $U$ normalizes $O'$. 

The intersection $O'\cap U$ is again finite, hence $O'$ is a non-trivial discrete subgroup of $G$. The group $U$ normalizers $O'$, whereby each $o\in O'$ has an open centralizer. The collection of elements with open centralizer in $G$, denoted $QZ(G)$, forms a normal subgroup. Since $G$ is abstractly simple, we conclude that $QZ(G)=G$. This is absurd since compactly generated totally disconnected locally compact groups with $QZ(G)$ dense have non-trivial compact normal subgroups via \cite[Proposition 4.3]{CM11}.
\end{proof}

\begin{cor} For $A_n\leq \Sym_n$ the alternating group with $n\geq 6$, every commensurated subgroup of $U(A_n)^+$ is either finite, compact and open, or $U(A_n)^+$.
\end{cor}

\bibliographystyle{amsplain}
\bibliography{biblio}

\def\cprime{$'$} \def\cprime{$'$}
\providecommand{\bysame}{\leavevmode\hbox to3em{\hrulefill}\thinspace}
\providecommand{\MR}{\relax\ifhmode\unskip\space\fi MR }
\providecommand{\MRhref}[2]{%
  \href{http://www.ams.org/mathscinet-getitem?mr=#1}{#2}
}
\providecommand{\href}[2]{#2}
\begin{thebibliography}{10}

\bibitem{MR0419616}
Hyman Bass, \emph{Some remarks on group actions on trees}, Comm. Algebra
  \textbf{4} (1976), no.~12, 1091--1126. \MR{0419616 (54 \#7634)}

\bibitem{BL89}
G.~M. Bergman and H.~W. Lenstra, Jr., \emph{Subgroups close to normal
  subgroups}, J. Algebra \textbf{127} (1989), no.~1, 80--97. \MR{1029404
  (91a:20004)}

\bibitem{MR2239037}
George~M. Bergman, \emph{Generating infinite symmetric groups}, Bull. London
  Math. Soc. \textbf{38} (2006), no.~3, 429--440. \MR{2239037 (2007e:20004)}

\bibitem{BM00}
M.~Burger and S.~Mozes, \emph{Groups acting on trees: from local to global
  structure}, Inst. Hautes \'Etudes Sci. Publ. Math. (2000), no.~92, 113--150
  (2001). \MR{1839488 (2002i:20041)}

\bibitem{CdM11}
Pierre-Emmanuel Caprace and Tom De~Medts, \emph{Simple locally compact groups
  acting on trees and their germs of automorphisms}, Transform. Groups
  \textbf{16} (2011), no.~2, 375--411. \MR{2806497 (2012m:22033)}

\bibitem{CM11}
Pierre-Emmanuel Caprace and Nicolas Monod, \emph{Decomposing locally compact
  groups into simple pieces}, Math. Proc. Cambridge Philos. Soc. \textbf{150}
  (2011), no.~1, 97--128. \MR{2739075 (2012d:22005)}

\bibitem{CER08}
Lisa Carbone, Mikhail Ershov, and Gordon Ritter, \emph{Abstract simplicity of
  complete {K}ac-{M}oody groups over finite fields}, J. Pure Appl. Algebra
  \textbf{212} (2008), no.~10, 2147--2162. \MR{2418160 (2009d:20067)}

\bibitem{MR2240370}
Yves de~Cornulier, \emph{Strongly bounded groups and infinite powers of finite
  groups}, Comm. Algebra \textbf{34} (2006), no.~7, 2337--2345. \MR{2240370
  (2007c:20098)}

\bibitem{C06}
\bysame, \emph{Strongly bounded groups and infinite powers of finite groups},
  Comm. Algebra \textbf{34} (2006), no.~7, 2337--2345. \MR{2240370
  (2007c:20098)}

\bibitem{zbMATH03981411}
John~D. {Dixon}, Peter~M. {Neumann}, and Simon {Thomas}, \emph{{Subgroups of
  small index in infinite symmetric groups.}}, {Bull. Lond. Math. Soc.}
  \textbf{18} (1986), 580--586 (English).

\bibitem{Dowerk:2015yu}
Philip~A. Dowerk and Andreas Thom, \emph{Bounded normal generation and
  invariant automatic continuity},  (2015), arXiv:1506.08549 [math.OA],
  http://arxiv.org/abs/1506.08549.

\bibitem{F70}
L{\'a}szl{\'o} Fuchs, \emph{Infinite abelian groups. {V}ol. {I}}, Pure and
  Applied Mathematics, Vol. 36, Academic Press, New York-London, 1970.
  \MR{0255673 (41 \#333)}

\bibitem{G00}
R.~I. Grigorchuk, \emph{Just infinite branch groups}, New horizons in pro-{$p$}
  groups, Progr. Math., vol. 184, Birkh\"auser Boston, Boston, MA, 2000,
  pp.~121--179. \MR{1765119 (2002f:20044)}

\bibitem{MR0072136}
Graham Higman, \emph{On infinite simple permutation groups}, Publ. Math.
  Debrecen \textbf{3} (1954), 221--226 (1955). \MR{0072136 (17,234d)}

\bibitem{K00}
Robert~R. Kallman, \emph{Every reasonably sized matrix group is a subgroup of
  {$S_\infty$}}, Fund. Math. \textbf{164} (2000), no.~1, 35--40. \MR{1784652
  (2001h:20005)}

\bibitem{K06}
Anatole Khelif, \emph{\`{A} propos de la propri\'et\'e de {B}ergman}, C. R.
  Math. Acad. Sci. Paris \textbf{342} (2006), no.~6, 377--380. \MR{2209214
  (2006j:20042)}

\bibitem{LOST10}
Martin~W. Liebeck, E.~A. O'Brien, Aner Shalev, and Pham~Huu Tiep, \emph{The
  {O}re conjecture}, J. Eur. Math. Soc. (JEMS) \textbf{12} (2010), no.~4,
  939--1008. \MR{2654085 (2011e:20016)}

\bibitem{N05}
Volodymyr Nekrashevych, \emph{Self-similar groups}, Mathematical Surveys and
  Monographs, vol. 117, American Mathematical Society, Providence, RI, 2005.
  \MR{2162164 (2006e:20047)}

\bibitem{N04}
Nikolay Nikolov, \emph{On the commutator width of perfect groups}, Bull. London
  Math. Soc. \textbf{36} (2004), no.~1, 30--36. \MR{2011975 (2004m:20055)}

\bibitem{NS12}
Nikolay Nikolov and Dan Segal, \emph{Generators and commutators in finite
  groups; abstract quotients of compact groups}, Invent. Math. \textbf{190}
  (2012), no.~3, 513--602. \MR{2995181}

\bibitem{RW_Hom_15}
Colin~D. Reid and Phillip~R. Wesolek, \emph{Homomorphisms into totally
  disconnected, locally compact groups with dense image}, (2015),
  arXiv:1509.00156 [math.GR], http://arxiv.org/abs/1509.00156.

\bibitem{R70}
Carl Riehm, \emph{The norm {$1$} group of a p-adic division algebra}, Amer. J.
  Math. \textbf{92} (1970), 499--523. \MR{0262250 (41 \#6860)}

\bibitem{MR2430439}
Christian Rosendal, \emph{Automatic continuity in homeomorphism groups of
  compact 2-manifolds}, Israel J. Math. \textbf{166} (2008), 349--367.
  \MR{2430439 (2009f:57055)}

\bibitem{R09}
\bysame, \emph{Automatic continuity of group homomorphisms}, Bull. Symbolic
  Logic \textbf{15} (2009), no.~2, 184--214. \MR{2535429 (2011e:03064)}

\bibitem{MR2503307}
\bysame, \emph{A topological version of the {B}ergman property}, Forum Math.
  \textbf{21} (2009), no.~2, 299--332. \MR{2503307 (2010a:03054)}

\bibitem{MR2365867}
Christian Rosendal and S{\l}awomir Solecki, \emph{Automatic continuity of
  homomorphisms and fixed points on metric compacta}, Israel J. Math.
  \textbf{162} (2007), 349--371. \MR{2365867 (2009d:54016)}

\bibitem{MR0476875}
Jean-Pierre Serre, \emph{Arbres, amalgames, {${\rm SL}_{2}$}}, Soci{\'e}t{\'e}
  Math{\'e}matique de France, Paris, 1977, Avec un sommaire anglais,
  R{{\'e}}dig{{\'e}} avec la collaboration de Hyman Bass, Ast{{\'e}}risque, No.
  46. \MR{0476875 (57 \#16426)}

\bibitem{SW13}
Yehuda Shalom and George~A. Willis, \emph{Commensurated subgroups of arithmetic
  groups, totally disconnected groups and adelic rigidity}, Geom. Funct. Anal.
  \textbf{23} (2013), no.~5, 1631--1683. \MR{3102914}

\bibitem{T99}
Simon Thomas, \emph{Infinite products of finite simple groups. {II}}, J. Group
  Theory \textbf{2} (1999), no.~4, 401--434. \MR{1718758 (2000m:20043)}

\bibitem{TZ12}
Simon {Thomas} and Jind\v{r}ich {Zapletal}, \emph{{On the Steinhaus and Bergman
  properties for infinite products of finite groups.}}, {Confluentes Math.}
  \textbf{4} (2012), no.~2, 1250002, 26 (English).

\bibitem{Tsankov:2011lr}
Todor Tsankov, \emph{Automatic continuity for the unitary group}, Proc. Amer.
  Math. Soc. \textbf{141} (2013), no.~10, 3673--3680. \MR{3080189}

\bibitem{W06}
Adam Woryna, \emph{On generation of wreath products of cyclic groups by two
  state time varying {M}ealy automata}, Internat. J. Algebra Comput.
  \textbf{16} (2006), no.~2, 397--415. \MR{2228520 (2007f:20073)}

\end{thebibliography}

\end{document}